\documentclass[aos]{imsart}
\usepackage{amsmath,amssymb,amsthm}
\usepackage[mathscr]{eucal}
\usepackage{mathtools}
\usepackage[math]{easyeqn}
\usepackage{srcltx}
\usepackage[numbers]{natbib} 
\usepackage{comment}
\usepackage{etoolbox}
\usepackage{mathrsfs}
\usepackage[usenames, dvipsnames]{color}
\usepackage{xfrac}
\usepackage{todonotes}
\usepackage{multirow}
\usepackage[autostyle]{csquotes}
\usepackage{booktabs}
\usepackage{longtable}
\usepackage{xr}
\usepackage{enumitem}
\usepackage{bbm}
\usepackage{bm}
\usepackage{url}
\renewcommand\cite{\citet}
\RequirePackage[colorlinks,linkcolor=blue,citecolor=blue,urlcolor=blue]{hyperref}

\doi{10.1214/22-AOS2192}

\newtheorem{lemma}{Lemma}[section]
\newtheorem{theorem}{Theorem}[section]
\newtheorem{proposition}{Proposition}[section]

\newtheorem{corollary}{Corollary}
\theoremstyle{remark}
\newtheorem{remark}{Remark}[section]
\theoremstyle{newremark}

\theoremstyle{remark}
\numberwithin{equation}{section}
\theoremstyle{definition}

\def \phid {\varphi}
\def \define{\coloneqq}

\def \Sigmaz {\Sigma_{z}}

\def \Sigmax {\Sigma}
\def \SigmaT {\Sigma_{T}}
\def \Sigmah {\hat{\Sigma}}
\def \Sigmac {\check{\Sigma}}

\def \epsc{\varepsilon}

\def \Var {\operatorname{Var}}

\def \tsum{{\textstyle \sum}}

\def\R{\mathbb{R}}
\def\E{\mathbb{E}}
\def\P{\mathbb{P}}

\newcommand\inner[2]{\langle #1, #2 \rangle}
\def \Id{I}
\def \Ind{\mathbbm{1}}
\def \dimp{p}

\def \cM {M}

\def \lmin{\lambda_{0}}
\def \lminz{\lambda_{z}}

\def \Zsigma{Z_{\Sigma}}

\def \xv{x}

\def \Zt {\tilde{Z}}
\def \Ut {\tilde{U}}

\def \St{\tilde{S}}

\def \st{\tilde{s}}

\def \aClbt {\tilde{C}_{\BCl}}

\def\betau{\beta_{u}}

\def \cmKK{\mathrm{b}}
\def \adj{{\mathrm{a}}}

\def \Xs{X^{\ast}}

\def \Ps{\Pb}
\def \Es{\Eb}

\def \Sbar {S_{0,n}}
\def \Snast {S^{\ast}_{n}}
\def \Xbar{\bar{X}}

\def \Xmean{\bar{X}}
\def \Xs{X^{\ast}}
\def \Pb{\P^{\ast}}
\def \Eb{\E^{\ast}}

\def \BCl{\mathscr{B}}
\def \ACl{\mathscr{A}}
\def \HCl{\mathscr{H}}
\def \Zy{Z}
\def \Uy{U}

\makeatletter
\let\@fnsymbol\@arabic
\makeatother
\makeatletter
\newcommand{\labitem}[2]{%
\def\@itemlabel{\textbf{#1}}
\item
\def\@currentlabel{#1}\label{#2}}
\makeatother

\begin{document}

\begin{frontmatter}

\title{New Edgeworth-type expansions with finite sample guarantees\thanksref{a1}}

\runtitle{Nonasymptotic Edgeworth-type expansions}
 \thankstext{A1}{First version submitted on June 5, 2020.}
 
\begin{aug}
\author{\fnms{Mayya} \snm{Zhilova}\thanksref{a2}}
\runauthor{M. Zhilova}

\thankstext{a2}{Support by the National Science Foundation Awards CAREER DMS-2048028 and DMS-1712990 is gratefully acknowledged.} 
\address{School of Mathematics\\ Georgia Institute of Technology\\
Atlanta, GA 30332-0160 USA\\
 e-mail: \href{mailto:mzhilova@math.gatech.edu}{\texttt{mzhilova@math.gatech.edu}}}
\end{aug}
\begin{abstract}
We establish higher-order nonasymptotic expansions for a difference between probability distributions of sums of i.i.d. random vectors in a Euclidean space. The derived bounds are uniform over two classes of sets: the set of all Euclidean balls and the set of all half-spaces. These results allow to account for an impact of higher-order moments or cumulants of the considered distributions; the obtained error terms depend on a sample size and a dimension explicitly. The new inequalities outperform accuracy of the normal approximation in existing Berry--Esseen inequalities under very general conditions. Under some symmetry assumptions on the probability distribution of random summands, the obtained results are optimal in terms of the ratio between the dimension and the sample size. The new technique which we developed for establishing nonasymptotic higher-order expansions can be interesting by itself.  
Using the new higher-order inequalities, we study accuracy of the nonparametric bootstrap approximation and propose a bootstrap score test under possible model misspecification. The results of the paper also include explicit error bounds for general elliptic confidence regions for an expected value of the random summands, and optimality of the Gaussian anti-concentration inequality over the set of all Euclidean balls. 
\end{abstract}
\begin{keyword}[class=MSC]
\kwd[Primary ]{62E17}
\kwd{62F40}
\kwd[; secondary ]{62F25}
\end{keyword}

\begin{keyword}
\kwd{Edgeworth series, dependence on dimension, higher-order accuracy, multivariate Berry--Esseen inequality, chi-square approximation, finite sample inference, anti-concentration inequality, bootstrap, elliptic confidence sets, linear contrasts, bootstrap score test, model misspecification}
\end{keyword}
\end{frontmatter}

\section{Introduction}

The Edgeworth series had been introduced by 
\citet{Edgeworth1896xi,Edgeworth1905law} and \citet{Chebyshev1890deux}, and developed by 
\citet{Cramer28} (see Section 2.9 by \citet{Hall1992bootstbook} for a detailed overview of early works about the Edgeworth series).
Since that time, the Edgeworth expansion has become one of the major asymptotic techniques for approximation of a c.d.f. or a p.d.f. 
In particular, the Edgeworth expansion is a powerful instrument for establishing rates of convergence in the CLT and for 
studying accuracy of the bootstrap.  

In this paragraph we recall a basic form of the Edgeworth series and their properties that are useful for comparison with the proposed results; this statement can be found in Chapter 5 by \cite{Hall1992bootstbook} (see also \cite{Bhattacharya1986normal,Kolassa2006series,Skovgaard1986multivariate}). 
 Let $S_{n}\define n^{-1/2}\sum_{i=1}^{n}X_{i}$ for i.i.d. $\R^{d}$-valued random vectors $\{X_{i}\}_{i=1}^{n}$ with $\E X_{i}=0, \Sigma\define \Var (X_{i})$, and $\E |X_{i}^{\otimes(k+2)}|<\infty$. Let $\ACl$ denote a class of sets $A\subseteq \R^{d}$ satisfying 
 \begin{align}
\label{cond:iso_intro}
 \sup\nolimits_{A\in\ACl} \int _{(\partial A)^{\epsc}} \phid(x)dx = O(\epsc), \quad \epsc \downarrow 0,
 \end{align}
where $ \phid(x)$ is the p.d.f. of $\mathcal{N}(0,\Id_{d})$, and $(\partial A)^{\epsc}$ denotes the set of points distant no more than $\epsc$ from the boundary $\partial A$ of $A$.  This condition holds for any measurable convex set in $\R^{d}$. Let also $\psi(t)\define \E e^{it^{T}X_{1}}$. If the Cram\'{e}r  condition 
\begin{align}
\label{ineq:CramerC}
\limsup\nolimits_{\|t\|\to\infty}|\psi(t)|<1
\end{align}
is fulfilled, then  
\begin{align}
\label{eq:ES_multivar_inro}
\P(S_{n}\in A)&=\int_{A} \{\phid_{\Sigma}(x)+{{\textstyle\sum\limits_{j=1}^{k}}} n^{-j/2} P_{j}(-\phid_{\Sigma}: \{\kappa_{j}\})(x) \bigr\}dx+o(n^{-k/2})
\end{align}
for $n\to \infty$.  The remainder term equals $o(n^{-k/2})$ uniformly in $A\in \ACl$, $\phid_{\Sigma}(x)$ denotes the p.d.f. of $\mathcal{N}(0,\Sigma)$; $ \kappa_{j}$ are cumulants of $X_{1}$, and $P_{j}(-\phid_{\Sigma}: \{\kappa_{j}\})(x)$ is a density of a signed measure, recovered from the series expansion of the characteristic function of $X_{1}$ using the inverse Fourier transform. In the multivariate case, a calculation of an expression for $P_{j}$ for large $j$ is rather involved since the number of terms  included in it grows with $j$ (see \cite{Mccullagh1987tensor}). 

Expansion \eqref{eq:ES_multivar_inro} does not hold for arbitrary random variables, in particular, Cram\'{e}r's condition \eqref{ineq:CramerC} holds if a probability distribution of $X_{1}$ has a nondegenerate absolutely continuous component. Condition \eqref{cond:iso_intro} does not take into account dependence on dimension $d$. Indeed, if $d$ is not reduced to a generic constant, then the right hand side of \eqref{cond:iso_intro} depends on $d$ in different ways for major classes of sets. Let  us refer to the works of \cite{Ball1993Reverse,Bentkus2003BE,Klivans2008lGaussSurface,Chernozhukov2012comparison,Belloni2018subvector}, where the authors established anti-concentration inequalities for important classes of sets.  

Due to the asymptotic form of the Edgeworth series \eqref{eq:ES_multivar_inro} for probability distributions, this kind  of expansions is typically used in the asymptotic framework (for $n\to \infty$) without  taking into account dependence of the remainder term $o(n^{-k/2})$ on the dimension. To the best of our knowledge, there have been no studies on accuracy of the Edgeworth expansions in finite sample multivariate setting so far. In this paper, we consider this framework and establish approximating bounds of type \eqref{eq:ES_multivar_inro} with explicit dependence on dimension $d$ and sample size $n$; this is useful for numerous contemporary applications, where it is important to track dependence of error terms on $d$ and $n$. Furthermore, these results allow to account for an impact of higher-order moments of the considered distributions, which is important for deriving approximation bounds with higher-order accuracy. In order to derive the explicit multivariate higher-order expansions, we propose a novel proof technique that can be interesting and useful by itself. 

One of the major applications of the proposed approximation bounds is the study of a performance of bootstrapping procedures in the nonasymptotic multivariate setting.  
In statistical inference, the bootstrap is one of the basic methods for estimation of probability distributions and quantiles of various statistics.  
Bootstrapping is well known for its \emph{good finite sample performance} (see, for example, \cite{Horowitz2001bootstrapHandbook}), for this reason it is widely used in applications. However, a majority of theoretical results about the bootstrap are asymptotic (for $n\to \infty$), and most of the existing works about bootstrapping in the nonasymptotic high-dimensional/multivariate setting are quite recent.  \cite{ArlotBlanch2010}  studied generalized weighted bootstrap for construction of nonasymptotic confidence bounds in $\ell_{r}$-norm for $r\in [1, +\infty)$ for the mean value of high dimensional random vectors with a symmetric and bounded (or with Gaussian) distribution.  \citet*{Chernozhukov2013CLT1} established results about accuracy of Gaussian approximation and  bootstrapping  for maxima of sums of high-dimensional vectors in a very general set-up. In \citep*{Chernozhukov2014CLT} the authors extended and improved the results from maxima to general hyperractangles and sparsely convex sets. Bootstrap approximations can provide \emph{faster rates of convergence} than the normal approximation (see \cite{PraestgaardWellner1993exchangeably,Barbe1995weighted,Liu1988,Mammen1993bootstrap,Lahiri2013resampling}, and references therein), however most of the existing results on this topic had been established in an asymptotic framework. In 2016, in \cite{Zhilova2020}, we derived higher-order properties of the nonparametric and multiplier bootstrap. Those results were the first progress made on the higher-order accuracy of the bootstrap in the nonasymptotic (and multivariate) framework.  
 In the present paper we derive new and much more general results. In particular, one of the implications of the proposed approximation bounds is an improvement of the Berry--Esseen inequality by \cite{Bentkus2003BE}. In Section \ref{sect:contrib} below we summarize the contribution and the structure of the paper.
\subsection{Contribution and structure of the paper}
\label{sect:contrib}
In Section \ref{sect:multivar} we establish expansions for the difference between probability distributions of $S_{n}\define n^{-1/2}\sum_{i=1}^{n}X_{i}$ for i.i.d. random vectors $\{X_{i}\}_{i=1}^{n}$ and $\mathcal{N}(0, \Sigma)$, $\Sigma\define \Var (S_{n})$. The bounds are uniform over two classes of subsets of $\R^{d}$: the set $\BCl$ of all $\ell_{2}$-balls, and the set $\HCl$ of all half-spaces. These classes of sets are useful when one works with linear or quadratic approximations of a smooth function of $S_{n}$; they are also useful for construction of confidence sets based on linear contrasts, for elliptic confidence regions, and for $\chi^{2}$-type approximations in various parametric models where a multivariate statistic is asymptotically normal. In Sections \ref{sect:ellip} and \ref{sect:score} we consider examples of elliptic confidence regions, Rao's score test for a simple null hypothesis, and its bootstrap version that remains valid even in case of a misspecified parametric model. 

In Theorem \ref{theorem:first}, where we study higher-order accuracy of the normal approximation of $S_{n}$ over the class $\BCl$, the 
approximation error is $\leq C n^{-1/2}R_{3}+ C\sqrt{d^{2}/n}+Cd^{2}/n$. $R_{3}$ is a sublinear function of the 3-d moment $\E (\Sigma^{-1/2} X_{1})^{\otimes 3}$, and $|R_{3}|\leq  \|\E (\Sigma^{-1/2} X_{1})^{\otimes 3}\|_{\mathrm{F}}$ for the Frobenius norm $ \|\cdot\|_{\mathrm{F}}$. 
The derived expressions for the error terms as well as the numerical constants are explicit. One of the implications of this result is an improvement of the Berry--Esseen inequality by \citet{Bentkus2003BE} that has the best known error rate for the class $\BCl$ (Remark \ref{remark:BE2} provides a detailed comparison between these results).

The proposed approximation bounds are not restricted to the normal approximation. In Theorems \ref{theorem:multivar_gen}, \ref{theorem:HS_var} we consider the uniform bounds between the distributions of $S_{n}$ and $S_{T,n}\define n^{-1/2}\sum_{i=1}^{n}T_{i}$ for i.i.d. random vectors $\{T_{i}\}_{i=1}^{n}$  with the same expected value as $X_{i}$ but possibly different covariance matrices. Here the error terms include a sublinear function of the differences  $\E(X_{1}^{\otimes j})- \E (T_{1}^{\otimes j})$ for $j=2,3$.

Let us  also emphasize that the derived expansions impose considerably weaker conditions on probability distributions of $X_{i}$ and $T_{i}$ than the Edgeworth expansions \eqref{eq:ES_multivar_inro} since our results do not require the Cram\'{e}r condition \eqref{ineq:CramerC} to be fulfilled, and they assume a smaller number of finite moments. Furthermore, the constants in our results do not depend on $d$ and $n$, which allows to track dependence of the error terms on them. To the best of our knowledge, there have been no such results obtained so far.

In Section \ref{sect:proof_contrib} we describe key steps of the proofs and the new technique which we developed for establishing the nonasymptotic higher-order expansions.

In Section \ref{sect:symm} we consider the case of symmetrically distributed $X_{i}$. The error term in the normal approximation bound is $\leq  C(d^{3/2}/n)^{1/2}$, which is smaller than the error term $\leq  C(d^{2}/n)^{1/2}$ provided by Theorem \ref{theorem:first} for the general case. Furthermore, we construct a lower bound, based on the example by \citet{Portnoy1986central}, showing that in this case the relation $d^{3/2}/n\to 0$ is required for consistency of the normal approximation. 

In Section \ref{section:ES_bootstr} we study accuracy of the nonparametric bootstrap approximation over set $\BCl$, using the higher-order methodology from Section \ref{sect:multivar}. The resulting error terms depend on  the quantities that characterize the sub-Gaussian tail behavior of $X_{i}$ (proportional to their  Orlicz $\psi_{2}$-norms) explicitly. 

In Section \ref{sect:cited} we collect statements from the earlier paper \citep{Zhilova2020} which are used in the proofs of main results; we also provide improved  bounds for constants in these statements and show optimality of the Gaussian anti-concentration bound over set $\BCl$. Proofs of the main results are presented in Sections \ref{sect:proof_2} and \ref{sect:ES_boostr_proofs}. 

\subsection{Notation}
For a vector $X=(x_{1},\dots,x_{d})^{T}\in \R^{d}$, $\|X\|$ denotes the Euclidean norm, $\E|X^{\otimes k}|<\infty$ denotes that $\E |x_{i_{1}} \cdots x_{i_{k}}|<\infty$ for all integer $i_{1},\dots, i_{k}\in \{1,\dots, d\}$. 
For tensors $A, B\in {\R^{d}}^{\otimes k}$, their inner product equals $\inner{A}{B}\define \sum_{1\leq i_{j}\leq d}a_{i_{1},\dots,i_{k}}b_{i_{1},\dots,i_{k}}$, where $a_{i_{1},\dots,i_{k}}$ and $b_{i_{1},\dots,i_{k}}$ are elements of $A$ and $B$.  
The operator norm of $A$ (for $k\geq 2$) induced by the Euclidean norm is denoted by $\|A\|\define \sup\{ \inner{A}{\gamma_{1}\otimes \dots \otimes \gamma_{k}}: \| \gamma_{j}\|=1, \gamma_{j}\in \R^{d}, j=1,\dots,k\}$ (see \cite{Wang2017operator}). 
The Frobenius norm is $\|A\|_{\mathrm{F}}=\sqrt{\inner{A}{A}}$. The maximum norm is $\|A\|_{\max}\define \max \{|a_{i_{1},\dots,i_{k}}|: i_{1},\dots, i_{k}\in \{1,\dots, d\}\}$. 
 For a function $f:\R^{d}\mapsto \R$ and $h\in \R^{d}$,  $f^{(s)}(x)h^{s}$ denotes the higher-order directional derivative $(h^{T}\nabla)^{s}f(x)$. $\phid(x)$ denotes  the p.d.f. of the standard normal distribution in $\R^{d}$. 
 $C,c$ denote positive generic constants. The abbreviations p.d. and p.s.d. denote positive definite and positive semi-definite matrices correspondingly.

\section{Higher-order approximation bounds}
\label{sect:multivar}
Denote for random vectors $X,Y$ in $\R^{d}$
\begin{align}
\label{def:distB}
\Delta_{\BCl}(X,Y)&\define \sup\limits_{ r\geq 0,\, t\in\R^{d}}\left|\P( \|X-t\|\leq r)-\P( \|Y-t\|\leq r)\right|.
\end{align}
Introduce the following functions
\begin{align}
\label{def:h1h2}
h_{1}(\beta)\define 
 h_{2}(\beta)+
 (1-\beta^{2})^{-1}\beta^{-4},
 \quad h_{2}(\beta)\define (1-\beta^{2})^{2}\beta^{-4}
\end{align}
for $\beta\in(0,1)$. 
Let $\{X_{i}\}_{i=1}^{n}$ be i.i.d. $\R^{d}$-valued random vectors with $\E |X_{i}^{\otimes 4}|<\infty$ and p.d. covariance matrix $\Sigma\define\Var (X_{i})$. Without loss of generality, assume that $\E X_{i}=0$. The following theorem provides the higher-order approximation bounds between $S_{n}\define n^{-1/2}\sum_{i=1}^{n}X_{i}$ and  the multivariate normal random vector $\Zsigma\sim \mathcal{N}(0,\Sigma)$ in terms of the distance $\Delta_{\BCl}(S_{n},\Zsigma)$.
\begin{theorem}
\label{theorem:first}
Suppose that the conditions above are fulfilled, then it holds for any $\beta\in(0,1)$
\begin{align}
\nonumber
&\Delta_{\BCl}(S_{n},\Zsigma)
\leq (\sqrt{6}\beta^{3})^{-1}R_{3}n^{-1/2}
\\
\nonumber
&  
\quad
+  2C_{B,4} \|\Sigma^{-1}\|\|\Sigma\|\bigl\{  
(h_{1}(\beta)+(4\beta^{4})^{-1})
\E\| \Sigma^{-1/2} X_{1}\|^{4}+d^{2}+2d
\bigr\}^{1/2}n^{-1/2}
   \\
   \nonumber
   &\quad+
   (2\sqrt{6})^{-1} \bigl\{
h_{1}(\beta) \E \|  \Sigma^{-1/2} X_{1}\|^4
   +h_2(\beta)(d^{2}+2d)\bigr\}n^{-1},
\end{align}
where $C_{B,4}\geq 9.5$ is a constant independent from $d,n$, and from a probability distribution of $X_{i}$ (see the definition of $C_{B,4}$ in the proof after formula \eqref{ineq:BE_bound}; 
$R_{3}$ is a sublinear function of $\E (\Sigma^{-1/2} X_{1})^{\otimes 3}$ such that, in general,
\begin{align*}
|R_{3}|&\leq \|\E (\Sigma^{-1/2} X_{1})^{\otimes 3}\|_{\mathrm{F}}
\leq \|\E (\Sigma^{-1/2} X_{1})^{\otimes 3}\| d.
\end{align*}
Furthermore,  if $N$ is the number of nonzero elements in $\E (\Sigma^{-1/2} X_{1})^{\otimes 3}$ and $N\leq d^{2}$, $m_{3}\define\|\E (\Sigma^{-1/2} X_{1})^{\otimes 3}\|_{\max}$, then $$|R_{3}|\leq m_{3}\sqrt{N}\leq m_{3} d$$ (a detailed definition of $R_{3}$ is given in \eqref{def:R3}). 
\end{theorem}
\begin{corollary}
\label{corollary:theorem1}
Let $\beta=0.829$, close to the local minimum of $h_{1}(\beta)$, then
\begin{align}
\nonumber
\Delta_{\BCl}(S_{n},\Zsigma)&
\leq
 0.717\|\E (\Sigma^{-1/2} X_{1})^{\otimes 3}\| dn^{-1/2}
\\
\nonumber
&\quad +
  2C_{B,4} \|\Sigma^{-1}\|\|\Sigma\|\bigl\{  
7.51
\E\|\Sigma^{-1/2} X_{1}\|^{4}+d^{2}+2d
\bigr\}^{1/2}n^{-1/2}
   \\
   \nonumber
   &\quad+
   \bigl\{
1.43 \E \|\Sigma^{-1/2} X_{1}\|^4
   +0.043(d^{2}+2d)\bigr\}n^{-1}.
 \end{align}
Let also $m_{4}\define\|\E (\Sigma^{-1/2} X_{1})^{\otimes 4}\|_{\max}$, then
\begin{align}
\label{ineq:DeltaB3}
\Delta_{\BCl}(S_{n},\Zsigma)
&\leq
  0.717\|\E (\Sigma^{-1/2} X_{1})^{\otimes 3}\| dn^{-1/2}
\\
\nonumber
&\quad  +
2 C_{B,4} \|\Sigma^{-1}\|\|\Sigma\|\bigl\{  
(7.51m_{4}+1)d^{2}+2d
\bigr\}^{1/2}n^{-1/2}
   \\&\quad+
   \bigl\{
(1.425m_{4}+0.043)d^{2}+0.086d\bigr\}n^{-1}.
   \nonumber
\end{align}
Hence, if all the terms in \eqref{ineq:DeltaB3} except $d$ and  $n$ are bounded by a generic constant $C>0$, then %
\begin{align}
\label{ineq:bound1}
\Delta_{\BCl}(S_{n},\Zsigma)   \leq C\bigl\{\sqrt{d^2/n} +d^2/n\bigr\}.
\end{align}
\end{corollary}
\begin{remark}
\label{remark:BE2}
The Berry--Esseen inequality by \citet{Bentkus2003BE} shows that for $\Sigma=\Id_{d}$  and $\E \|X_{1}\|^{3}<\infty$,  $\Delta_{\BCl}(S_{n},\Zsigma)\leq c \E \|X_{1}\|^{3} n^{-1/2}$. The term $(\sqrt{6}\beta^{3})^{-1}n^{-1/2}R_{3}$ in Theorem \ref{theorem:first} has an explicit constant and, since this
is a sublinear function of the third moment of $\Sigma^{-1/2}X_{1}$, it can be considerably smaller than the third moment of the $\ell_{2}$-norm $\|\Sigma^{-1/2}X_{1}\|$. Corollary \ref{corollary:theorem1} shows that the error term in Theorem \ref{theorem:first} depends on $d$ and $n$ as $C(\sqrt{d^2/n} + d^2/n)$, which \emph{improves the Berry--Esseen approximation error} $C\sqrt{d^3/n}$ in terms of the ratio between $d$ and $n$. Theorem \ref{theorem:first} imposes a stronger moment assumption than the Berry--Esseen bound by \cite{Bentkus2003BE}, since the latter inequality assumes only 3 finite moments of $\|X_{i}\|$. However, the theorems considered here require much weaker conditions than the Edgeworth expansions \eqref{eq:ES_multivar_inro} that would assume in general at least $5$ finite moments of $\|X_{i}\|$ and the Cram\'{e}r condition  \eqref{ineq:CramerC}.
\end{remark}
\begin{remark}
Since functions $h_{1}(\beta), h_{2}(\beta)$ are known explicitly \eqref{def:h1h2}, the expression of the approximation error term in Theorem \ref{theorem:first} contains explicit constants and it even allows to optimize the error term (w.r.t.  parameter $\beta\in (0,1)$), depending on $R_{3}$, $\E \|\Sigma^{-1/2} X_{1}\|^4$, and other terms as well. Therefore, the results in this paper allow to address the problem of finding an optimal constant in Berry--Esseen inequalities (see, for example, \cite{Shevtsova2011absolute}).
\end{remark}

The following statement is an extension of Theorem \ref{theorem:first} to a general (not necessarily normal) approximating distribution.  Let $\{T_{i}\}_{i=1}^{n}$  be i.i.d random vectors in $\R^{d}$, with $\E T_{i}=0$, p.d. covariance matrix $\Var (T_{i})=\SigmaT$, and $\E |T_{i}^{\otimes 4}|<\infty$. Let also $S_{T,n}\define n^{-1/2} \sum_{i=1}^{n}T_{i}$.
\begin{theorem}
\label{theorem:multivar_gen}
Let $\{X_{i}\}_{i=1}^{n}$ satisfy conditions of Theorem \ref{theorem:first}. Firstly consider the case $\Var(T_{i})=\Var (X_{i})=\Sigma$. Denote $$\bar{V}_{4}\define \E\|\Sigma^{-1/2}X_{1}\|^{4}+\E\|\Sigma^{-1/2}T_{1}\|^{4}.$$ It holds for any $\beta\in(0,1)$
\begin{align*}
\Delta_{\BCl}(S_{n},S_{T,n})
&\leq (\sqrt{6}\beta^{3})^{-1}\bar{R}_{3,T}n^{-1/2}
+   (2\sqrt{6})^{-1} 
h_{1}(\beta) \bar{V}_{4}
   n^{-1},
\\
\nonumber
&  
\quad
+  \sqrt{8} C_{B,4} \|\Sigma^{-1}\|\|\Sigma\|\bigl\{  
(h_{1}(\beta)+(4\beta^{4})^{-1})\bar{V}_{4}
+2d^{2}+4d
\bigr\}^{1/2}n^{-1/2},
\end{align*}
where  $\bar{R}_{3,T}$ is a sublinear function of $\E (\Sigma^{-1/2} X_{1})^{\otimes 3}-\E (\Sigma^{-1/2} T_{1})^{\otimes 3}$ 
such that, in general,
\begin{align*}
|\bar{R}_{3,T}|&\leq \|\E (\Sigma^{-1/2} X_{1})^{\otimes 3}-\E (\Sigma^{-1/2} T_{1})^{\otimes 3} \|_{\mathrm{F}}
\\&\leq \|\E (\Sigma^{-1/2} X_{1})^{\otimes 3}-\E (\Sigma^{-1/2} T_{1})^{\otimes 3}\| d.
\end{align*}
Furthermore,  if $N_{T}$ is the number of nonzero elements in $\E (\Sigma^{-1/2} X_{1})^{\otimes 3}-\E (\Sigma^{-1/2} T_{1})^{\otimes 3}$, and $N_{T}\leq d^{2}$, $m_{3,T}=\|\E (\Sigma^{-1/2} X_{1})^{\otimes 3}-\E (\Sigma^{-1/2} T_{1})^{\otimes 3}\|_{\max}$, then $$|\bar{R}_{3,T}|\leq m_{3,T}\sqrt{N_{T}}\leq m_{3,T} d.$$

Now consider the case when $\Var X_{i}=\Sigma$ and $\Var T_{i}=\SigmaT$ are not necessarily equal to each other. Let  $\lmin^{2}>0$ denote the minimum of the smallest eigenvalues of $\Sigma$ and $\SigmaT$. Denote 
$$V_{4}\define \E\|X_{1}\|^{4}+\E\|T_{1}\|^{4}, \quad v_{4}\define \|\Sigmax\|^{2}+\|\SigmaT\|^{2}.$$ It holds for any $\beta\in(0,1)$ 
\begin{align*}
\Delta_{\BCl}(S_{n},S_{T,n})
&\leq (\sqrt{2}\beta^{2}\lmin^{2})^{-1}\|\Sigmax-\SigmaT\|_{\mathrm{F}}+(\sqrt{6}\beta^{3})^{-1}R_{3,T}n^{-1/2}
\\
\nonumber
&  
\quad
+  4\sqrt{2}C_{B,4}\lmin^{-2}\bigl\{  
h_{1}(\beta)V_{4}
+(d^{2}+2d)(v_{4}+1/2)
\bigr\}^{1/2}n^{-1/2}
\\&\quad+   2(\sqrt{6}\lmin^{4})^{-1}\bigl\{h_{1}(\beta) V_{4} +(d^{2}+2d)v_{4} \bigr\}n^{-1}, 
\end{align*}
where  ${R}_{3,T}$ is a sublinear function of $\E (X_{1}^{\otimes 3})-\E (T_{1}^{\otimes 3})$ 
such that, in general,
\begin{align*}
|{R}_{3,T}|&\leq \lmin^{-3}\|\E (X_{1}^{\otimes 3})-\E (T_{1}^{\otimes 3})\|_{\mathrm{F}}
\leq \lmin^{-3}\|\E (X_{1}^{\otimes 3})-\E (T_{1}^{\otimes 3})\| d
\end{align*}
 (a detailed definition of $R_{3,T}$ is given in \eqref{def:R3_T}. 
\end{theorem}
Below we consider the uniform distance between the probability distributions of  $S_{n}$ and $\Zsigma$ over the set of all half-spaces in $\R^{d}$:
\begin{align}
\label{def:distH}
 \Delta_{\HCl}(S_{n},\Zsigma)&\define \sup\limits_{ x\in\R, \gamma \in\R^{d}}\left|\P( \gamma^{T}S_{n} \leq x)-\P(\gamma^{T}\Zsigma\leq x )\right|.
\end{align}
Denote $h_{3}(\beta)\define3\beta^{-4}\{1-(1-\beta^{2})^{2}\}$ for $\beta\in(0,1)$ (similarly to $h_{1},h_{2}$ introduced in \eqref{def:h1h2}).
\begin{theorem}
\label{theorem:first_HS}
Given the conditions of Theorem \ref{theorem:first}, it holds $\forall\beta\in(0,1)$
\begin{align*}
 \Delta_{\HCl}(S_{n},\Zsigma)
&\leq 
(\sqrt{6}\beta^{3})^{-1} \| \E (\Sigma^{-1/2} X_{1})^{\otimes 3}\| n^{-1/2} 
\\
&\quad +C_{H,4} \bigl\{ (h_{1}(\beta)+\beta^{-4})\| \E (\Sigma^{-1/2} X_{1})^{\otimes 4}\|+ h_{3}(\beta)
\bigr\}^{1/2}n^{-1/2}
\\&\quad
 +  
   (2\sqrt{6})^{-1}\{
h_{1}(\beta)\| \E (\Sigma^{-1/2} X_{1})^{\otimes 4}\| +3h_{2}(\beta)
 \}n^{-1},
\end{align*}
where $C_{H,4}\geq 9.5$ is a constant independent from $d,n$, and a probability distribution of $X_{i}$ (see the definition of $C_{H,4}$ in the proof after formula \eqref{ineq:Delta_th_2}).
\end{theorem}
\begin{corollary}
\label{corollary:theorem_H1}
Let $\beta=0.829$, close to the local minimum of $h_{1}(\beta)$, then
\begin{align*}
 \Delta_{\HCl}(S_{n},\Zsigma)
&\leq
  0.717\| \E (\Sigma^{-1/2} X_{1})^{\otimes 3}\| n^{-1/2} 
\\&\quad+ 
C_{H,4}\bigl\{  
9.10\| \E (\Sigma^{-1/2} X_{1})^{\otimes 4}\| +5.731 \bigr\}^{1/2}n^{-1/2}
\\&\quad+  
   \{
1.425\| \E (\Sigma^{-1/2} X_{1})^{\otimes 4}\| +0.127%
 \}n^{-1}.
\end{align*}
\end{corollary}
\begin{remark}
The inequalities that we establish for the class $\HCl$ are \emph{dimension-free}. Indeed, the approximation errors in Theorems \ref{theorem:first_HS}, \ref{theorem:HS_var} and Corollary \ref{corollary:theorem_H1} depend only on numerical constants, on $n^{-1/2}, n^{-1}$, and on the operator norms of the 3-d and the 4-th moments of $\Sigma^{-1/2} X_{1}$: 
$$\| \E (\Sigma^{-1/2} X_{1})^{\otimes j}\| = \sup\limits_{\gamma\in \R^{d}, \|\gamma\|=1} \E (\gamma^{T} \Sigma^{-1/2} X_{1})^{j}\quad \text{for }j=3,4.$$
\end{remark}
\begin{remark}
\label{remark:HS_4}
Recalling the arguments in Remark \ref{remark:BE2}, $\|\E (\Sigma^{-1/2} X_{1})^{\otimes 3}\|$ in the latter statement depends on the third moment of $X_{1}$ sublinearly.  Furthermore, the classical Berry--Esseen theorem  by \citet{Berry1941accuracy} and \citet{Esseen1942Liapounoff} (that requires $\E|X_{i}^{\otimes 3}|<\infty$) gives an error term $\leq c \| \E (\Sigma^{-1/2} X_{1})^{\otimes 4}\|^{3/4}n^{-1/2}$ which is $\geq \sqrt{\| \E (\Sigma^{-1/2} X_{1})^{\otimes 4}\|/n}$ because $ \| \E (\Sigma^{-1/2} X_{1})^{\otimes 4}\|\geq 1$. This justifies that Theorem \ref{theorem:first_HS} can have a better accuracy than the result for $ \Delta_{\HCl}$ implied by the classical Berry--Esseen inequality when, for example, $\E X_{1}^{\otimes 3}=0$ and $\| \E (\Sigma^{-1/2} X_{1})^{\otimes 4}\|$ is rather big  (e.g., for the logistic or von Mises distributions).
\end{remark}

The following statement extends Theorem \ref{theorem:first_HS} to the case of a general (not necessarily normal) approximating distribution, similarly to Theorem \ref{theorem:multivar_gen}. Recall that $v_{4}= \|\Sigmax\|^{2}+\|\SigmaT\|^{2}$ and denote $$\bar{V}_{T,4}\define\| \E (\Sigma^{-1/2} X_{1})^{\otimes 4}\|+\| \E (\Sigma^{-1/2} T_{1})^{\otimes 4}\|,\quad  {V}_{T,4}\define\| \E (X_{1}^{\otimes 4})\|+\| \E (T_{1}^{\otimes 4})\|.$$ 
\begin{theorem}
\label{theorem:HS_var}Given the conditions of Theorem \ref{theorem:multivar_gen}, it holds $\forall \beta\in(0,1)$
\begin{align*}
 \Delta_{\HCl}(S_{n},S_{T,n})
&\leq
(\sqrt{6}\beta^{3})^{-1} \|\E (\Sigma^{-1/2} X_{1})^{\otimes 3} -\E (\Sigma^{-1/2} T_{1})^{\otimes 3}\| n^{-1/2} 
\\&\quad+ 
C_{H,4}\bigl\{  
(h_{1}(\beta)+\beta^{-4})\bar{V}_{T,4}+ 2h_{3}(\beta)
\bigr\}^{1/2}n^{-1/2}
\\&\quad+  
   (2\sqrt{6})^{-1}
h_{1}(\beta)\bar{V}_{T,4}n^{-1}.
\end{align*}

Consider the case when $\Var X_{i}=\Sigma$ and $\Var T_{i}=\SigmaT$ are not necessarily equal to each other. Let  $\lmin^{2}>0$ denote the minimum of the smallest eigenvalues of $\Sigma$ and $\SigmaT$.  It holds $\forall\beta\in(0,1)$
\begin{align*}
\Delta_{\HCl}(S_{n},S_{T,n})
&\leq (\sqrt{2}\beta^{2}\lmin^{2})^{-1}\|\Sigmax-\SigmaT\|
\\&\quad+(\sqrt{6}\beta^{3}\lmin^{3})^{-1}\|\E (X_{1}^{\otimes 3})-\E (T_{1}^{\otimes 3})\|n^{-1/2}
\\
\nonumber
&  
\quad
+  4\sqrt{2}C_{H,4}\lmin^{-2}\bigl\{  
h_{1}(\beta)V_{T,4}
+3(v_{4}+1/2)
\bigr\}^{1/2}n^{-1/2}
\\&\quad+   2(\sqrt{6}\lmin^{4})^{-1}\bigl\{h_{1}(\beta) V_{T,4} +3v_{4} \bigr\}n^{-1}.
\end{align*}
\end{theorem}

\section{New proof technique}
\label{sect:proof_contrib}

In this section we describe the key steps and ideas that we develop in the proofs of Theorems \ref{theorem:first} and \ref{theorem:multivar_gen} for the class $\BCl$. Theorems \ref{theorem:first_HS} and \ref{theorem:HS_var} about half-spaces are derived in an analogous way.

First we use the triangle inequality 
\begin{align}
\label{ineq: triangle_intro}
\Delta_{\BCl}(S_{n},\Zsigma)&\leq \Delta_{\BCl}(S_{n},\St_{n})+\Delta_{\BCl}(\St_{n},\Zsigma),
\end{align}
where $\St_{n}=n^{-1/2}\sum\nolimits_{i=1}^{n}Y_{i}$ for i.i.d. random summands $Y_{i}$ constructed in a special way (see definitions \eqref{def:Y_i}, \eqref{eq:Y_moments}). We define $Y_{i}$ such that 
they have matching moments of orders $1,2,3$ with the original random vectors $X_{i}$, and at the same time they have a normal component which plays crucial role in the smoothing technique that we describe below. 

 To the term $\Delta_{\BCl}(S_{n},\St_{n})$ in \eqref{ineq: triangle_intro} we apply the Berry--Esseen type inequality from \citep{Zhilova2020} (this result is discussed in more detail in Section \ref{sect:cited}, 
 which yields 
 \begin{align}
\nonumber
 \Delta_{\BCl}(S_{n},\St_{n})&
 \\&\hspace{-0.8cm}
 \leq2C_{B,4} \|\Sigma^{-1}\|\|\Sigma\|
\sqrt{
 \bigl\{  
(h_{1}(\beta)+0.25/\beta^{4})
\E\| \Sigma^{-1/2} X_{1}\|^{4}+d^{2}+2d
\bigr\}/n
}.
\label{ineq:BEbound}
\end{align}
 Here the error rate $C \sqrt{d^2/n}$ comes from the higher-order moment matching property between the random summands of $S_{n}$  and $\St_{n}$, which improves the ratio $C \sqrt{d^3/n}$ between dimension $d$ and sample size $n$ in the classical Berry--Esseen result by \cite{Bentkus2003BE} (in the classical Berry--Esseen theorem one uses, in general, only first two matching moments which is smaller than first three moments). 
Also the square root in this expression naturally comes from the smoothing argument used for derivation of the Berry--Esseen inequality with the best-known rate w.r.t. $d$ and $n$, and it is unavoidable for the distance $\Delta_{\BCl}(S_{n},\St_{n})$ under the mild conditions on $X_{i}$ imposed here. The proof of the result in \citep{Zhilova2020} is based on an extension of the proof of the Berry--Esseen inequaliy by \cite{Bentkus2003BE}. 

 For the term $\Delta_{\BCl}(\St_{n},\Zsigma)$ in \eqref{ineq: triangle_intro}, we exploit the structure of $\St_{n}$ in order to construct the higher-order expansion that allows to compare moments of $X_{i}$ and $\Zsigma$. It holds 
 \begin{equation}
 \St_{n} \overset{d}{=} \tilde{Z}+n^{-1/2}\sum\nolimits_{i=1}^{n}U_{i},
 \label{eq:gpart}
\end{equation}
  where $\tilde{Z} \sim \mathcal{N}(0,\beta^{2}  \Sigma)$ for $\beta\in (0,1)$ that enters the resulting bounds as a free parameter and can be used for optimizing the approximation error terms w.r.t. it. Random vectors $\{U_{i}\}_{i=1}^{n}$ are i.i.d. and independent from $\tilde{Z}$, and $\{X_{i}\}_{i=1}^{n}$, hence the expression in \eqref{eq:gpart} has multivariate normal distribution, conditionally on $\{U_{i}\}$.  Also 
 \begin{align}
 \label{eq:moments3}
 \E U_{i}= \E X_{i} =0, \quad \Var U_{i}= (1-\beta^{2}) \Sigma,  \quad \E U_{i}^{\otimes 3} = \E X_{i}^{\otimes 3}. 
 \end{align}
We introduce the following representation of the probability distribution functions of $\St_{n}$ and $\Zsigma$. Let $B\in \BCl$ and $B^{\prime}\define \{x\in \R^{d}: \beta\Sigma^{1/2} x\in B\}$, and $Z_{0}\define \beta^{-1} \Sigma^{-1/2} \tilde{Z} \sim \mathcal{N}(0,\Id_{d})$, then it holds
 \begin{align}
 \label{eq:gpart_2}
 \P(\St_{n} \in B) &= \E\bigl\{ \P\bigl(\tilde{Z}+n^{-1/2}\tsum_{i=1}^{n}U_{i}\in B \vert \{U_{i}\}_{i=1}^{n}\bigr)\bigr\}
 \\&= \E\bigl\{ \P\bigl(Z_{0}+n^{-1/2}\beta^{-1}\tsum_{i=1}^{n}\Sigma^{-1/2}U_{i}\in B^{\prime} \vert \{U_{i}\}_{i=1}^{n}\bigr)
 \bigr\}
\nonumber
\\&= \E \int\nolimits_{B^{\prime}}\varphi(t-n^{-1/2}\beta^{-1}\tsum_{i=1}^{n}\Sigma^{-1/2}U_{i}) dt,
\nonumber
\end{align}
for $\varphi(t)$ denoting the p.d.f. of $Z_{0}$. In this way, \emph{we use the normal component} of $\St_{n}$ to represent $\P(\St_{n} \in B)$ \emph{as an expectation of a smooth function} of the sum if i.i.d random vectors $n^{-1/2}\beta^{-1}\tsum_{i=1}^{n}\Sigma^{-1/2}U_{i}$ that have \emph{matching moments} with the original samples $X_{i}$. The same representation holds for the approximating distribution $\Zsigma$. Let $Z_{i}\sim \mathcal{N}(0, (1-\beta^{2})\Sigma)$ be i.i.d., independent from all other random vectors with the same first two moments as $U_{i}$, then
$\Zsigma\overset{d}{=}\tilde{Z}+n^{-1/2}\sum_{i=1}^{n}Z_{i}$,
 \begin{align}
  \label{eq:gpart_3}
  \P(\Zsigma \in B)&= \E\Bigl\{ \P\bigl(\tilde{Z}+n^{-1/2}\tsum_{i=1}^{n}Z_{i}\in B \vert \{Z_{i}\}_{i=1}^{n}\bigr)\Bigr\}
   \\&= \E\Bigl\{ \P\bigl(Z_{0}+n^{-1/2}\beta^{-1}\tsum_{i=1}^{n}\Sigma^{-1/2}Z_{i}\in B^{\prime} \vert \{Z_{i}\}_{i=1}^{n}\bigr)\Bigr\}
   \nonumber
   \\&= \E \int\nolimits_{B^{\prime}}\varphi(t-n^{-1/2}\beta^{-1}\tsum_{i=1}^{n}\Sigma^{-1/2}Z_{i}) dt.
\nonumber
\end{align}
Now we represent the difference 
$\P(\St_{n} \in B)-\P(\Zsigma \in B)$
as the following telescoping sum (the general telescopic sum principle or the swapping method is due to \cite{Lindeberg1922neue}, see also \cite{Trotter1959elementary} and \cite{Chatterjee2006generalization}):
 \begin{align}
 \nonumber
& \P(\St_{n} \in B)-\P(\Zsigma \in B)
\\&=\sum\nolimits_{i=1}^{n}\E \int\nolimits_{B^{\prime}}
\bigl\{\varphi(t-(n^{1/2}\beta\Sigma^{1/2})^{-1}U_{i}- s_{i})
 \label{tel_sum}
\\&\hspace{2.75cm} -\varphi(t-(n^{1/2}\beta\Sigma^{1/2})^{-1}Z_{i}- s_{i}) \bigr\} dt,
 \nonumber
 \end{align}
where $s_{i}=n^{-1/2}\beta^{-1}\sum\nolimits_{k=1}^{i-1}\Sigma^{-1/2} Z_{k}+n^{-1/2}\beta^{-1}\sum\nolimits_{k=i+1}^{n}\Sigma^{-1/2}U_{k}$ for $i=1,\dots,n$, the sums are taken equal zero if index $k$ runs beyond the specified range.  Random vectors $s_{i}$ are independent from $U_{i}$ and $Z_{i}$ which is used in the proof together with the Taylor expansion of $\varphi(t)$
and the matching moments property \eqref{eq:moments3}. Further details of the calculations are available in Section \ref{sect:proof_2}. 
The resulting error bound 
 \begin{align*}
\Delta_{\BCl}(\St_{n},\Zsigma)&
\leq (\sqrt{6}\beta^{3})^{-1}R_{3}n^{-1/2}
\\
   \nonumber
   &\quad+
   (2\sqrt{6})^{-1} \bigl\{
h_{1}(\beta) \E \|  \Sigma^{-1/2} X_{1}\|^4
   +h_2(\beta)(d^{2}+2d)\bigr\}n^{-1},
 \end{align*}
is fully explicit, nonasymptotic, and is analogous to the terms in the classical Edgeworth series. 

The proof of Theorem \ref{theorem:multivar_gen} uses an analogous approach. First we write the triangle inequality
\begin{align}
\label{ineq:triangle_intro_T}
\Delta_{\BCl}(S_{n},S_{T,n})&\leq \Delta_{\BCl}(S_{n},\St_{n})+ \Delta_{\BCl}(S_{T,n}, \tilde{S}_{T,n}) +\Delta_{\BCl}(\St_{n}, \tilde{S}_{T,n}),
\end{align}
where $\tilde{S}_{T,n}= n^{-1/2}\sum\nolimits_{i=1}^{n}Y_{T,i}$ is constructed similarly to the approximating sum $\St_{n}$ (see \eqref{eq:Stn_def} -- \eqref{def:Ui_T} for their explicit definitions). The terms  $\Delta_{\BCl}(S_{n},\St_{n})$, $\Delta_{\BCl}(S_{T,n}, \tilde{S}_{T,n})$ are bounded similarly to  $\Delta_{\BCl}(S_{n},\St_{n})$ in \eqref{ineq:BEbound}, and the term $\Delta_{\BCl}(\St_{n}, \tilde{S}_{T,n})$ is expanded in the same way as $\Delta_{\BCl}(\St_{n},\Zsigma)$ using the smooth normal components, the telescoping sum representations, and the Taylor series expansions.

Let us emphasize that the proposed proof technique is much more simple than many existing methods of deriving rates of convergence in the normal approximation. Furthermore, it is not restricted to the case when an approximation distribution is normal and it allows to obtain explicit error terms and constants under very mild conditions. To the best of our knowledge, this is a novel technique  and it has not been used in earlier literature.

\begin{remark}
I submitted the first version of this paper containing the new proof technique to \emph{The Annals of Statistics} on June 5, 2020 (submission number AOS2006-011). In about 3 months after that \cite{lopes2020} (see \url{https://arxiv.org/abs/2009.06004v1}) used a very similar approach and labeled it as \enquote{implicit Gaussian smoothing}. \cite{lopes2020} considered the problem of bounding the distance 
\begin{equation}
\label{dist_rect}
\Delta_{\mathscr{R}}(S_{n},Z_{\Sigma})=\sup\limits_{R\in \mathscr{R}}\abs{ \P(S_{n}\in R)-\P(Z_{\Sigma}\in R)},
\end{equation}
where $ \mathscr{R}$ is a set of all hyperrectagles in $\R^{d}$.  
  The approach used by \cite{lopes2020} has major similarities with the approach that we presented earlier in this section. These ideas play crucial role for my solution in the present paper as well as for the proofs in \cite{lopes2020}. Below we describe these similarities:
\begin{enumerate}[label={(\roman*)}]
\item 
\cite{lopes2020} uses the normal part of a random sum similarly to \eqref{eq:gpart} in order to represent its probability distribution of a sum of independent random vectors as an expected value of a smooth function (via the Gaussian distribution) similarly to \eqref{eq:gpart_2}, \eqref{eq:gpart_3}. 
\item The smooth function obtained via the Gaussian distribution in part (i) is expanded using the Taylor series as in \eqref{tel_sum} (see also  \eqref{eq:tel}--\eqref{eq:proof1}
 in Section \ref{sect:proof_2}; 
 \cite{lopes2020} uses the 2-nd order Taylor expansion with the remainder in the same form as here (see \eqref{eq:R4_in}, \eqref{eq:Taylor}). The Taylor expansion  is applied to the differences in the telescoping sum similarly to \eqref{tel_sum} and \eqref{eq:tel}--\eqref{eq:proof1}.
 
\item  The expansion in (ii) allows to apply the moment matching strategy as in \eqref{eq:moments_1}, \eqref{eq:moments_2}.
\item \cite{lopes2020} uses the induction technique based on the proof by \cite{Bentkus2003BE}, similarly to our approach for \eqref{ineq:BEbound}. 
\end{enumerate}
 I would like to emphasize that the approach presented here as well as the ideas from our 2016 paper \citep*{Zhilova2020} are new and first appeared in my work, previous to \cite{lopes2020}.
\end{remark}

\section{Approximation bounds for symmetric distributions and optimality of the error rate}
\label{sect:symm}
In this section we consider the case when the probability distribution of $X_{i}-\E X_{i}$ has some symmetry properties this  assumption can be formulated in terms of the condition  \eqref{condit:poly}  on moments of $X_{i}$.
Suppose for  i.i.d. $\R^{d}$-valued random vectors $\{X_{i}\}_{i=1}^{n}$ that $\E |X_{i}^{\otimes 6}|<\infty$ and their covariance matrix $\Sigma\define \Var (X_{i})$ is p.d. Without loss of generality, assume that $\E X_{i}=0$. Let $X=(x_{1},\dots, x_{d})$ be an i.i.d. copy of $X_{i}$, we assume that 
\begin{align}
\label{condit:poly}
\E p(x_{1},\dots,x_{d})=0
\end{align}
for any monomial $p: \R^{d}\mapsto \R$ that has degree $\leq 5$ and contains an odd power of $x_{j}$ for at least one  $j\in\{1,\dots,d\}$.  In addition, we assume that there exist a random vector $U_{L}$ in $\R^{d}$ with $\E |U_{L}^{\otimes 6}|<\infty$ and a p.d. covariance matrix $\Sigma_{L}\in \R^{d\times d}$ such that the following moment matching property holds for $Z_{\Sigma, L}\sim \mathcal{N}(0, \Sigma_{L})$ independent from $U_{L}$ and  $L\define Z_{\Sigma_{L}}+U_{L}$:
\begin{align}
\label{condit:L_moments}
 \E(L^{\otimes j})& = \E(X_{i}^{\otimes j})\quad \forall j \in \{1,\dots, 5\}.
\end{align}
We introduced this condition  in earlier paper \citep{Zhilova2020}; Lemmas 3.1, 3.2 in that paper show that under certain conditions on the support of $X_{i}$ there exists a probability distribution $\mathcal{L}(L)$ which complies with these conditions (see Lemma \ref{lemma:moment_conditions} in Section \ref{sect:cited} for further details). Also, because of property \eqref{condit:poly}, it is sufficient to assume that there exist only 6 finite absolute moments (instead of 7 finite absolute moments as stated in Lemma 3.1 in \citep{Zhilova2020} for the general case).
\begin{theorem}
\label{theorem:symm}
Let $\{X_{i}\}_{i=1}^{n}$ follow the conditions above, take $\lminz^{2}>0$ equal to the smallest eigenvalue of $\Sigma_{L}$, and  $\Zsigma\sim \mathcal{N}(0,\Sigma)$ in $\R^{d}$, then it holds %
\begin{align*} 
\Delta_{\BCl}(S_{n},\Zsigma)&
\leq C_{B,6}\bigl\{ 
\lminz^{-6} 
\E(\|X_{1}\|^{6}+\|L_{1}\|^{6})
\bigr\}^{1/4}n^{-1/2}
\\& \quad+
 (4!)^{-1/2}\lminz^{-4}\|\E (X_{1}^{\otimes 4}) - \E(\Zsigma^{\otimes 4}) \|_{\mathrm{F}}  n^{-1} 
 \\&\quad+  ({6!})^{-1/2}\lminz^{-6}\bigl\{  \E \|U_{L}\|^6+\E\|\Zsigma\|^{6}\bigr\}n^{-2},
\end{align*}
where  $C_{B,6}=2.9 C_{\ell_{2}}C_{\phi,6}\geq 2.9 $ is a constant independent from $d,n$, and probability distribution of $X_{i}$ (it is discussed in detail in Section \ref{sect:cited}). 
Let $m_{6,sym}$ denote the maximum of the 6-th moments of the coordinates of $X_{1},L_{1}, \Zsigma$, then the above inequality implies
\begin{align}
\nonumber 
\Delta_{\BCl}(S_{n},\Zsigma)&\leq  (4!)^{-1/2}\lminz^{-4}\|\E (X_{1}^{\otimes 4}) - \E(\Zsigma^{\otimes 4}) \|_{\mathrm{F}}  n^{-1} 
 \\
 \nonumber
 &\quad 
 +
 C_{B,6}(
\lminz^{-6} 
m_{6,sym}
)^{1/4}d^{3/4}n^{-1/2}+
 (6!)^{-1/2}(
\lminz^{-6} 
m_{6,sym}
)d^{3}n^{-2}
\\
\label{ineq:kappa4_sym}
&\leq   8^{-1/2} \lminz^{-4}\|\E (X_{1}^{\otimes 4}) - \E(\Zsigma^{\otimes 4})\|_{\max} d n^{-1} 
 \\
 \nonumber
 &\quad 
+C_{B,6}(
\lminz^{-6} 
m_{6,sym}
)^{1/4}d^{3/4}n^{-1/2}+(6!)^{-1/2}(
\lminz^{-6} 
m_{6,sym}
)d^{3}n^{-2}.
\end{align}
\end{theorem}

Below we consider the example by  \citet{Portnoy1986central}  (Theorem 2.2 in \citep{Portnoy1986central}), using the notation in the present paper, and we derive a lower bound for $\Delta_{\BCl}(S_{n},\Zsigma)$ with the ratio between $d$ and $n$ similar to the error term in Theorem \ref{theorem:symm}. Proposition \ref{prop:example_d} and Lemma \ref{lemma:appl} imply that for $\{X_{i}\}_{i=1}^{n}$, $Z$  as in Theorem \ref{theorem:Portnoy}, and for sufficiently large $d,n$ 
$$ C d^{3/2}/n \leq \Delta_{\BCl}(S_{n},Z)  \leq  C(d^{3/2}/n)^{1/2}.$$
\begin{theorem}[\citet{Portnoy1986central}]
\label{theorem:Portnoy}
Let i.i.d.  random vectors $X_{i}$ have the following mixed normal distribution
\begin{align*}
X_{i}\rvert \{Z_{i}\} &\sim \mathcal{N}(0,Z_{i}Z_{i}^{T})\quad \text{for i.i.d. } Z_{i}\sim \mathcal{N}(0,\Id_{d}).
\end{align*}
Let also $S_{n}=n^{-1/2}\sum_{i=1}^{n} X_{i}$, $Z\sim \mathcal{N}(0,\Id_{d})$. If $d\to \infty$ such that $d/n\to 0$ as $n\to \infty$, then
\begin{gather*}
 \|S_{n}\|^{2}=\|Z\|^{2}  +D_{n},\quad
D_{n}=O_{p}(d^{2}/n).
\end{gather*}
\end{theorem}
\begin{proposition}
\label{prop:example_d}
Let $\{X_{i}\}_{i=1}^{n}$ and $Z$ be as in Theorem \ref{theorem:Portnoy}. If $d\to \infty$ such that $d/n\to 0$ as $n\to \infty$, then 
$$\Delta_{\BCl}(S_{n},Z)  \geq  \Delta_{L}( \|S_{n}\|^{2}, \|Z\|^{2}) = O (d^{3/2}/n),$$
where  $\Delta_{L}(X,Y)\define \inf \{\varepsilon>0: G(x-\varepsilon)-\varepsilon\leq F(x)\leq G(x+\varepsilon)+\varepsilon \text{ for all } x\in \R\}$ denotes the  L\'evy distance between the c.d.f.-s of $X$ and $Y$, equal $F(x)$ and $G(x)$ respectively.
\end{proposition}
\begin{lemma}
\label{lemma:appl}
$\{X_{i}\}_{i=1}^{n}$  in Theorem \ref{theorem:Portnoy}  satisfy conditions of Theorem \ref{theorem:symm} for $\lminz=(1-\sqrt{2/5})^{1/2}$.
\end{lemma}
\section{Bootstrap approximation}
\label{section:ES_bootstr}
Here we consider the nonparametric or Efron's bootstrapping scheme for $S_{n}$ (by \citet{Efron1979bootstrap,Efron1994introduction})  and study its accuracy in the framework of Theorems \ref{theorem:first} and \ref{theorem:multivar_gen}. 
Let $\{X_{i}\}_{i=1}^{n}$ be i.i.d. $\R^{d}$-valued random vectors with $\E |X_{i}^{\otimes 4}|<\infty$, p.d. $\Sigma\define\Var X_{i}$ and $\mu\define \E X_{i}$. Introduce resampled variables $\Xs_{1},\dots,\Xs_{n}$ with zero mean, according to the nonparametric bootstrapping scheme: 
\begin{equation}
\label{eq:EB}
\Pb(\Xs_{i}=X_{j}-\Xmean)=1/n\quad \forall i,j=1,\dots,n,
\end{equation}
 where $\Xmean=n^{-1}\sum_{i=1}^{n}X_{i}$ and  $\Pb(\cdot)\define\P(\cdot\vert \{X_{i}\}_{i=1}^{n})$, $\Es(\cdot)\define \E(\cdot\vert \{X_{i}\}_{i=1}^{n})$. Hence $\{\Xs_{j}\}_{j=1}^{n}$ are i.i.d. and $$\Eb(\Xs_{j})=\E(X_{i}-\mu)=0,\quad  
\Es({\Xs_{j}}^{\otimes k})=n^{-1}{{\textstyle\sum\nolimits_{i=1}^{n}}} (X_{i}- \Xmean)^{\otimes k} \approx_{\P} \E (X_{i}-\mu)^{\otimes k}$$ for $k\geq 1$; the sign $\approx_{\P}$ denotes \enquote{closeness} with high probability. Denote the centered sum $S_{n}$ and its bootstrap approximation as follows
$$\Sbar\define n^{-1/2}\sum\nolimits_{i=1}^{n} (X_{i}-\mu),\quad \Snast  \define  n^{-1/2}\sum\nolimits_{i=1}^{n} \Xs_{i}.$$
In order to quantify the accuracy of the bootstrap approximation of the probability distribution of $\Sbar$, we compare the empirical moments $n^{-1}{{\textstyle\sum\nolimits_{i=1}^{n}}} (X_{i}- \Xmean)^{\otimes k}$ and the population moments $\E (X_{i}-\mu)^{\otimes k}$ for $k=2,3$, using exponential concentration bounds. For this purpose we introduce condition \eqref{condit:expm1} on tail behavior of coordinates of $X_{i}-\mu$. Here we follow the notation from Section 2.3 by \cite{BouchLug2013Conc}. A random real-valued variable $x$ belongs to class $\mathcal{G}(\sigma^{2})$ of sub-Gaussian random variables with variance factor $\sigma^{2}>0$ if
\begin{align}
\label{condit:expm1}
\E \bigl\{\exp (\gamma x)\bigr\}&\leq \exp\bigl(\gamma^{2}\sigma^{2}/2\bigr) \quad \forall \gamma \in \R.
\end{align}
We assume that every coordinate of random vectors $X_{i}-\mu$, $i=1,\dots,n$ belongs to the class $\mathcal{G}(\sigma^{2})$ for some $\sigma^{2}>0$. 
Let also 
\begin{gather}
\label{def:tast_Cast}
\begin{split}
C_{1}(t) \define 2 \{4\sqrt{2t} + 3tn^{-1/2}\}, \quad C_{2}(t)\define 4\sqrt{2}(\sqrt{8}t+t^{3/2}n^{-1/2}),
\\
 t_{\ast}\define \log n+\log(2 dn +d^{2}+3d), \quad C_{j \ast}\define C_{j}(t_{\ast}) \text{ for } j=1,2.
 \end{split}
\end{gather}
$\lambda_{\min}(\Sigma)>0$ denotes the smallest eigenvalue of the covariance matrix $\Sigma$. We consider  $d$ and $n$ such that
\begin{gather}
\label{condit:lambda_n}
\sigma^{2}C_{1\ast}d/\sqrt{n}
< \lambda_{\min}(\Sigma).
\end{gather}
This condition allows to ensure that the approximation bound in Theorem \ref{theorem:bootst} holds with high probability. Recall that $h_{1}(\beta)= 
(1-\beta^{2})^{2}\beta^{-4}+
 (1-\beta^{2})^{-1}\beta^{-4}$ for $\beta\in (0,1)$
\begin{theorem}
\label{theorem:bootst}
If the conditions introduced above are fulfilled, then it holds with probability $\geq 1-n^{-1}$ 
\begin{align}
\nonumber
\Delta_{\BCl}(\Sbar,\Snast)
&\leq \delta_{\BCl},
\end{align}
where $\delta_{\BCl} =\delta_{\BCl}(d,n, \mathcal{L}(X_{i}))$ is defined as follows
\begin{align}
\label{ineq:deltaB_2}
\delta_{\BCl} &
\define 
(\sqrt{2}\beta^{2}\lmin^{2})^{-1}\bigl\{  \sigma^{2}C_{1\ast} d/\sqrt{n}\bigr\}
\\
\label{ineq:deltaB_3_1}
&\quad
+(\sqrt{6}\beta^{3} \lmin^{3})^{-1}
\Bigl[ 
4\sigma \sqrt{2 dn^{-2} t_{\ast}}
\{ \|\Sigma\|_{\mathrm{F}}+\sigma^{2}t_{\ast}d/n\}
\\
\label{ineq:deltaB_3_2}
&\quad+
\sigma^{2}d^{3/2}n^{-1} C_{2\ast}\{1+3n^{-1/2}\} +\|\E( X_{1}-\mu)^{\otimes 3}\|_{\mathrm{F}} n^{-1/2}
\Bigr]
\\
\nonumber
&  
\quad+  4\sqrt{2}C_{B,4}\lmin^{-2}\Bigl\{  
h_{1}(\beta) \bigl[ \E \|X_{1}-\mu\|^{4}+  8(1+n^{-2})\{2\sigma^{2}t_{\ast}d/n \}^{2} \bigr]
\\
\nonumber
&
\quad+(d^{2}+2d)( 3\|\Sigma\|^{2} + 2 \bigl\{ \sigma^{2}C_{1\ast}d/\sqrt{n}\bigr\}^{2}+1/2)
\Bigr\}^{1/2}n^{-1/2}
\\
\nonumber
&
\quad+    2(\sqrt{6}\lmin^{4})^{-1}\Bigl\{h_{1}(\beta)\bigl[ \E \|X_{1}-\mu\|^{4}+8(1+n^{-2})\{2\sigma^{2} t_{\ast} d/n \}^{2} \bigr]
\\
\nonumber
& 
\quad+(d^{2}+2d)\bigl[ 3\|\Sigma\|^{2} +  2 \bigl\{ \sigma^{2}C_{1\ast}d/\sqrt{n}\bigr\}^{2}\bigr] \Bigr\}n^{-1}
\end{align}
 for arbitrary $\beta\in (0,1)$ and for $\lmin^{2}\define  \lambda_{\min}(\Sigma)- \sigma^{2}C_{1\ast}d/\sqrt{n}$. 
 \end{theorem}
\begin{remark}
\label{remark:delta}
The explicit approximation error $\delta_{\BCl}$ in Theorem \ref{theorem:bootst} allows to evaluate accuracy of the bootstrap in terms of $d,n, \sigma^{2}$, and moments of $X_{i}$. In general, $\delta_{\BCl} \leq C_{\ast} \{\sqrt{d^2/n} +d^2/n\}$
(for $C_{\ast}$ depending on the $\log$-term $t_{\ast}$ and moments of $X_{i}$), however the expression for  $\delta_{\BCl}$  provides a much more detailed and accurate characterization of the approximation error.  The proof of this result is based on the second statement in Theorem \ref{theorem:multivar_gen} (for $\Sigma$ and $\SigmaT$ not necessarily equal to each other). The first term on the right-hand side of \eqref{ineq:deltaB_2} and  the summands in \eqref{ineq:deltaB_3_1}, \eqref{ineq:deltaB_3_2}  characterize the distances between the population moments $ \E (X_{i}-\mu)^{\otimes k}$ and their consistent estimators $\Es({\Xs_{j}}^{\otimes k})$ (for $k=2$ and 3 respectively). The rest of the summands in the expression for $\delta_{\BCl}$  correspond to the higher-order remainder terms which leads to smaller error terms for a sufficiently large $n$. 
\end{remark}

\section{Elliptic confidence sets}
\label{sect:ellip}
 An elliptic confidence set is one of the major types of confidence regions in statistical theory and applications. They are commonly constructed for parameters of (generalized) linear regression models, in ANOVA methods, and in various parametric models where a multivariate statistic is asymptotically normal. As for example, in the case of the score function considered in Section \ref{sect:score}. See, for instance, \cite{Friendly2013elliptical} for an overview of statistical problems and applications involving elliptic confidence regions.

In this section we construct confidence regions for an expected value of i.i.d. random vectors $\{X_{i}\}_{i=1}^{n}$,  using the bootstrap-based quantiles. %
Since the considered set-up is rather general, the provided results can be used for various applications, where one is interested in estimating a mean of an observed sample in a nonasymptotic and multivariate setting. See, for example, \cite{ArlotBlanch2010}, where the authors constructed nonasymptotic confidence bounds in $\ell_{r}$-norm for the mean value of high dimensional random vectors and considered a number of important practical applications.

Let $W \in \R^{d\times d}$ be a p.d. symmetric matrix. $W$ is supposed to be known, it defines the quadratic form of an elliptic confidence set:
$B_{W}(x_{0},r)\define \{x \in \R^{d}:  (x-x_{0})^{T}W(x-x_{0})\leq r\},$ 
for $x_{0}\in \R^{d}, r \geq 0$.
There are various ways of how one can interpret and use $W$ in statistical models. For example, $W$ can serve for weighting an impact of residuals in linear regression models in the presence of errors' heteroscedasticity (cf. weighted least squares estimation); for  regularized least squares estimators in the linear regression model (for example, ridge regression) $W$ denotes a regularized covariance matrix of the LSE; see \citep*{Friendly2013elliptical} for further examples.

In Proposition \ref{prop:ellip2} below we construct an elliptic confidence region for $\E X_{1}$ based on the bootstrap approximation established in  Section \ref{section:ES_bootstr}.  Let $\Xbar^{\ast}\define n^{-1}\sum_{j=1}^{n}\Xs_{j}$ for the i.i.d. bootstrap sample $\{\Xs_{j}\}_{j=1}^{n}$ generated from the empirical distribution of $\{X_{i}-\Xbar\}_{i=1}^{n}$ for $\Xbar= n^{-1}\sum_{i=1}^{n}X_{i}$. Let  also
$$q^{\ast}_{\alpha}\define \inf\{t>0: (1-\alpha) \leq \Pb( n^{1/2}\|W^{1/2} \Xbar^{\ast} \| \leq t)\}$$ denote $(1-\alpha)$-quantile of the bootstrap statistic $n^{1/2}\|W^{1/2} \Xbar^{\ast} \|$ for arbitrary $\alpha \in (0,1)$.  
We assume that  coordinates of vectors $\{W^{1/2}(X_{i}-\E X_{i})\}_{i=1}^{n}$ are sub-Gaussian with variance factor $\sigma_{W}^{2}>0$ (i.e. condition \eqref{condit:expm1} is fulfilled).  %
Let also $d,n$ be such that
$ \sigma_{W}^{2}C_{1\ast}d/\sqrt{n}
< \lambda_{\min}(W^{1/2}\Sigma W^{1/2})$ (for $C_{1\ast}$ defined in \eqref{def:tast_Cast}). Theorem \ref{theorem:bootst} implies the following statement
\begin{proposition}
\label{prop:ellip2}
If the conditions above are fulfilled, it holds
\begin{align*}
\bigl|\P\bigl( n^{1/2}\|W^{1/2} (\Xbar - \E X_{1})\|  \leq q^{\ast}_{\alpha} \bigr) - (1-\alpha)\bigr| &\leq \delta_{W}.
\end{align*}
$\delta_{W}$ is analogous to $\delta_{\BCl}$, where we take $\Sigma \define W^{1/2} \Sigma W^{1/2}$, $\sigma^{2}\define \sigma_{W}^{2}$, etc. A detailed definition of $\delta_{W}$ is given in \eqref{def:deltaW}, see also Remark \ref{remark:delta} for the discussion about its dependence on $d$ and n.
\end{proposition}
\section{Score tests}
\label{sect:score}
Let $y=(y_{1},\dots,y_{n})$ be an i.i.d. sample from a p.d.f. or a p.m.f. $p(x)$. Let also $\mathcal{P}\define\{p(x;\theta): \theta \in \Theta \subseteq \R^{d}\}$ denote a known parametric family of  probability distributions. The unknown function $p(x)$ does not necessarily belong to the parametric family $\mathcal{P}$, in other words the parametric model can be misspecified. Following the renown aphorism of \cite{Box1976} \enquote{All models are wrong, but some are useful}, it is widely recognized that in general a (parametric) statistical model cannot be considered exactly correct. See, for example, \cite{White82,Gustafson2001,Wit2012all}, \textsection 1.1.4 by \cite{GLMbook1983}, and p. 2 by \cite{Bickel2015mathematical_II}. Hence it is of particular importance to design methods of statistical inference that are \emph{robust to model misspecification}. In this section we propose \emph{a bootstrap score test procedure} which is valid even in case when the parametric model $\mathcal{P}$ is misspecified.

Let $s(\theta)=s(\theta,y)$ and $I(\theta)$ denote the score function and the Fisher information matrix corresponding to the introduced parametric model
\begin{align*}
s(\theta)\define \sum\nolimits_{i=1}^{n}{\partial} \log p(y_{i};\theta)/{\partial \theta} ,\quad
I(\theta)\define \Var\{ s(\theta)\}.
\end{align*}
We suppose that the standard regularity conditions on the parametric family $\mathcal{P}$ are fulfilled. Let $\theta_{0}\define\mathrm{argmin}_{\theta\in \Theta}\E\log (p(y_{i})/p(y_{i};\theta))$ denote the parameter which corresponds to the  projection of $p(x)$ on the parametric family $\mathcal{P}$ w.r.t. the Kullback-Leibler divergence (also known as the relative entropy).

Consider a simple hypothesis $H_{0}: \theta_{0}=\theta^{\prime}$.  Rao's score test (by \cite{Rao1948large}) for testing $H_{0}$ is based on the following test statistic and its convergence in distribution to $\chi^{2}_{d}$
\begin{equation}
\label{eq:Rtest}
R(\theta^{\prime})\define s(\theta^{\prime})^{T}\{I(\theta^{\prime})\}^{-1} s(\theta^{\prime}) \overset {d \vert H_{0}}{\to} \chi^{2}_{d},\quad n\to \infty,
\end{equation}
provided that matrix $I(\theta^{\prime})$ is p.d. The sign $\overset {d \vert H_{0}}{\to}$ denotes convergence in distribution under $H_{0}$.  
  Matrix $I(\theta^{\prime})$ can be calculated explicitly for a known $\theta^{\prime}$ if one assumes that $p(x)\in \mathcal{P}$, i.e. if the parametric model is correct. However, if $p(x)$ does not necessarily belong to the considered parametric class $\mathcal{P}$, then neither $I(\theta^{\prime})$ nor the probability distribution of $s(\theta^{\prime})$ can be calculated in an explicit way under the general assumptions considered here. In this case, the Fisher information matrix $I(\theta)$ is typically estimated using the Hessian of the log-likelihood function $\sum_{i=1}^{n}\log p(y_{i};\theta)$. However the standardization with an empirical Fisher information may considerably reduce the power of the score test for a small sample size $n$ (see \cite{Rao2005} and \cite{Freedman2007can}).

Below we consider {a bootstrap score test} for testing simple hypothesis $H_{0}$, under possible misspecification of the parametric model. Denote 
$$\tilde{R}(\theta^{\prime})\define \|s(\theta^{\prime})/\sqrt{n}\|^{2}.$$
 One can consider $s(\theta^{\prime})=\sum\nolimits_{i=1}^{n}X_{i},$ where random vectors $X_{i}\define \partial \log p(y_{i};\theta^{\prime})/{\partial \theta^{\prime}}$ are i.i.d. with $\E X_{i} =0$ under $H_{0}$. Introduce the bootstrap approximations of $s(\theta^{\prime})$ and $\tilde{R}(\theta^{\prime})$:
$$s^{\ast}(\theta^{\prime})\define \sum\nolimits_{i=1}^{n}\Xs_{i},\quad   R^{\ast}(\theta^{\prime})\define \|s^{\ast}(\theta^{\prime})/\sqrt{n}\|^{2},$$
where $\{\Xs_{i}\}_{i=1}^{n}$ are sampled according to Efron's bootstrap scheme \eqref{eq:EB}. Let also 
$$t^{\ast}_{\alpha}\define \inf\{t>0: (1-\alpha) \leq \Pb( R^{\ast}(\theta^{\prime}) \leq t)\}$$ denote $(1-\alpha)$-quantile of the bootstrap score statistic for arbitrary $\alpha \in (0,1)$. Suppose that coordinates of vectors $X_{i}= \partial \log p(y_{i};\theta^{\prime})/{\partial \theta^{\prime}}$ satisfy condition \eqref{condit:expm1} with variance factor $\sigma^{2}_{s}>0$. Let also $d$ and $n$ be such that $
\sigma_{s}^{2}C_{1\ast} d/\sqrt{n} <  \lambda_{\min}(I(\theta^{\prime}))/n$.
 Then Theorem \ref{theorem:bootst} implies the following statement which characterizes accuracy of the bootstrap score test under $H_{0}$. 
\begin{theorem}[Bootstrap score test]
\label{theorem:score1}
If the conditions above are fulfilled, it holds
\begin{align*}
\bigl|\P_{H_{0}}\bigl(\tilde{R}(\theta^{\prime}) > t^{\ast}_{\alpha} \bigr) -\alpha\bigr| &\leq \delta_{R},
\end{align*}
where $\delta_{R}$ is analogous to $\delta_{\BCl}$ up to the terms $\sigma^{2}$ and $\Sigma$. A detailed definition of $\delta_{R}$ is given in \eqref{def:deltaR}, see also Remark \ref{remark:delta} for the discussion about its dependence on $d$ and n.
\end{theorem}

The following statement provides a finite sample version of Rao's score test based on \eqref{eq:Rtest} for testing simple hypothesis $H_{0}: \theta_{0}=\theta^{\prime}$. Here we require also the finite 4-th moment of the score in order to apply the higher-order approximation from Theorem \ref{theorem:first}.  
\begin{theorem}[Nonasymptotic version of Rao's score test]
\label{theorem:score2}
Suppose that $p(x)\equiv p(x,\theta_{0})$ for some $\theta_{0}\in \Theta$, i.e. there is no misspecification in the considered parametric model. 
Let also $\tilde{X}_{i}\define \sqrt{n} \{I(\theta^{\prime})\}^{-1/2}\partial \log p(y_{i};\theta^{\prime})/\partial \theta^{\prime}$ denote the marginal standardized score for the $i$-th observation and $\tilde{\Sigma}\define n^{-1} I(\theta^{\prime})$.  Suppose that $\E |\tilde{X}_{i}^{\otimes 4}|<\infty$, 
then the asymptotic poperty \eqref{eq:Rtest}
 for testing $H_{0}: \theta_{0}=\theta^{\prime}$ can be represented in the finite sample form as follows
\begin{align*}
&\sup\nolimits_{\alpha\in(0,1)}\bigl|\P_{H_{0}}\bigl(R(\theta^{\prime}) > q(\alpha;\chi^{2}_{d}) \bigr) -\alpha\bigr| \leq 
 (\sqrt{6}\beta^{3})^{-1} \|\E (\tilde{X}_{1}^{\otimes 3})\|_{\mathrm{F}}n^{-1/2}
\\
\nonumber
&  
\quad
+  2C_{B,4} \|\tilde{\Sigma}^{-1}\|\|\tilde{\Sigma}\|\bigl\{  
(h_{1}(\beta)+(4\beta^{4})^{-1})
\E\|\tilde{X}_{1}\|^{4}+d^{2}+2d
\bigr\}^{1/2}n^{-1/2}
   \\
   \nonumber
   &\quad+
   (2\sqrt{6})^{-1} \bigl\{
h_{1}(\beta) \E \|\tilde{X}_{1}\|^4
   +h_2(\beta)(d^{2}+2d)\bigr\}n^{-1},
\end{align*}
where $q(\alpha;\chi^{2}_{d})$ denotes the $(1-\alpha)$-quantile of $\chi^{2}_{d}$ distribution. The inequality holds for any $\beta\in  (0,1)$, functions $h_{1},h_{2}$ are defined in \eqref{def:h1h2}, constant $C_{B,4}\geq 9.5$ is described in the statement of Theorem \ref{theorem:first}. 
\end{theorem}
%

%
\section{Statements used in the proofs}
\label{sect:cited}
Theorem \ref{theorem:BE16} and Lemma \ref{lemma:AC_GSA} are used in the proofs of the main results; these statements had been derived in our earlier paper \citep{Zhilova2020} (its first version appeared in 2016). Here we provide improved lower bounds for constant $M$ (in Theorem \ref{theorem:BE16}), and we also describe the constants in the error term which appear in the main results. In Remark \ref{ref:AC_opt} we show optimality (with respect to $d$) of the Gaussian anti-concentration bound over set $\BCl$. Lemma \ref{lemma:moment_conditions} is a concise version of Lemmas 3.1 and 3.2 in \citep{Zhilova2020}, which provide conditions on distribution $X_{i}$ that are sufficient for fulfilling condition \eqref{eq:moments_2} of Theorem \ref{theorem:BE16} (as well as condition \eqref{condit:L_moments} of Theorem \ref{theorem:symm}).

Let random vectors $\{X_{i}\}_{i=1}^{n}$ in $\R^{d}$ be i.i.d. and such that $\E |X_{i}^{\otimes K}|<\infty$ for some integer $K\geq 3$,  and   $\Var(X_{i}) = \Sigma$ is p.d. Suppose that there exist i.i.d. approximating random vectors $\{Y_{i}\}_{i=1}^{n}$ such that $\E |Y_{i}^{\otimes K}|<\infty$,
\begin{equation}
\label{eq:moments_2}
\begin{gathered}
\E (X_{i}^{\otimes j})=\E (Y_{i}^{\otimes j})\ \forall j=1,\dots,K-1,\text{ and}\\
 Y_{i}=\Zy_{i}+\Uy_{i} \text{ for independent r.v. $\Zy_{i},\Uy_{i}\in\R^{d}$ s.t.,}\\ 
\text{$\Zy_{i}\sim \mathcal{N}(0,\Sigmaz)$ for a p.d. $\Sigmaz$.}
\end{gathered} 
\end{equation}
Denote $S_{n}\define n^{-1/2}{ \sum\nolimits_{i=1}^{n}} X_{i}$, $\St_{n}\define n^{-1/2}{{\textstyle \sum\nolimits_{i=1}^{n}}} Y_{i},$  and let $\lminz^{2}>0$ be equal to the smallest eigenvalue of the covariance matrix $\Sigmaz$.

\begin{theorem}
\label{theorem:BE16}
Let $\{X_{i}\}_{i=1}^{n}$ and $\{Y_{i}\}_{i=1}^{n}$ meet the conditions above. Suppose, without loss of generality, that $\E X_{i}=\E Y_{i}=0$, then it holds
\begin{align*}
\Delta_{\BCl}(S_{n},\tilde{S}_{n})
&\leq  C_{B,K}{ \left\{\lminz^{-K}\E\left(\|X_{1}\|^{K}+\|Y_{1}\|^{K}\right)\right\}^{{1}/{(K-2)}}}{n^{-1/2}},
\end{align*}
where constant $C_{B,K}=M(K) C_{\ell_{2}}C_{\phi,K}$ depends only on $K$. For the quantity $M(K)$, one can take $M(3)=54.1$, $M(4)=9.5$, $M(6)=2.9$.
\end{theorem}
\noindent
 $C_{\ell_{2}}\define\max\{1,\aClbt\}$ for a numeric constant $\aClbt$ given in Lemma \ref{lemma:AC_GSA} further in this section.  
 $C_{\phi,K}\define\max\{1,{C}_{\phi,1,K-2}, {C}_{\phi,1,K-1}\}$ is specified as follows: fix $\epsc>0$ and a Euclidean ball $B= B(x_{0},r)\define \left\{x \in \R^{\dimp}: \|x-x_{0}\|\leq r\right\}$ in $\R^{d}$, let function $\psi: \R^{d}\mapsto \R$ be defined as $\psi(x;B)\define \phi(\tilde{\rho}(x;B)/\tilde{\epsc})$, where $\tilde{\rho}(x;B)= \{\|x-x_{0}\|^{2}-r^{2}\}\Ind\{x\notin B\}$, $\tilde{\epsc}=\epsc^{2}+2r\epsc$, and $\phi(x)$ is a sufficiently smooth approximation of a step function 
 \begin{gather*}
0\leq \phi(x)\leq 1, \quad 
\phi(x)=
\begin{cases}
1,& x\leq 0;\\
0,& x\geq 1.
\end{cases}
 \end{gather*}
 Constants ${C}_{\phi,1,j}>0$ for $j=K-2,K-1$ are such that $\forall B\in \BCl$, $\forall\,x,h\in \R^{d}$ 
$$\bigl|\psi^{(j)}(x;B)h^{j}\bigr|\leq {{C}_{\phi,1,j} \|h\|^{j}}{(\epsc/2)^{-j}}\Ind\{x\in B(x,r+\epsc)\setminus B\}.$$ 
Further details about these terms are provided in Lemma A.3 in  \citep{Zhilova2020supp}.

The smaller values of $M=M(K)$ (compared to $M\geq 72.5$ in  \citep{Zhilova2020}) are obtained by optimizing the following expression w.r.t.  $(\adj,\cmKK,M)$ (this expression is contained in the last inequality in the end of the proof of Theorem 2.1 in \citep{Zhilova2020supp}):
\begin{align*}
 \frac{\adj}{\cM}
 +
 \frac{1.5 (\adj/2)^{-(K-2)} }{(K-2)!}\frac{(2\cM+\adj)}{\cM}
 \nonumber
+
 \frac{4\sqrt{2}\cmKK }{(\adj/2)^{(K-1)}} 
\frac{(2\cM+\adj)}{\cM K! }&
\\+\frac{2\sqrt{2} }{\sqrt{K!}\cmKK^{K-2}}
+  2^{(K-3)/(K-2)}\frac{ 2.6}{M^{K-2}}
& \leq 1,
\end{align*}
we take $(K,M,\adj,\cmKK)\in\{(3, 54.1, 27.46, 14), (4,9.5, 6.33, 8.5), (6,2.9, 2.07, 8.5)\}$.

In the following lemma we study the anti-concentration properties of $Z\sim \mathcal{N}(0,\Sigma)$ over the class $\BCl$. Similar bounds for the standard normal distribution were considered by \cite{Sazonov1972,Ball1993Reverse,Klivans2008lGaussSurface}). 
Let $B=B(x_{0},r)$ and $\epsc>0$. Denote  $B^{\epsc}\define B(x_{0},r+\epsc).$
\begin{lemma}[Anti-concentration inequality for $\ell_{2}$-balls in $\R^{d}$]
\label{lemma:AC_GSA}
Let $Z\sim \mathcal{N}(\mu,\Sigma)$ for arbitrary $\mu\in\R^{d}$ and p.d. $\Sigma$ with the smallest eigenvalue $\lmin^{2}>0$. It holds for any $\epsc>0$ and for a numeric constant $\aClbt>0$
\begin{align}
\nonumber
\sup\nolimits_{B\in\BCl}
\P\bigl(Z\in B^{\epsc}\setminus B\bigl)
&\leq   \epsc\aClbt/\lmin.
\end{align}
\end{lemma}

\begin{remark}
\label{ref:AC_opt}
This inequality is sharp w.r.t. dimension $d$. Indeed, consider $Z\sim \mathcal{N}(0,\Id_{d})$, $B=B(0,r)$, and $\epsc>0$, Let $f_{\chi_{d}}(t)$ denote the p.d.f. of $\|Z\|\sim \chi_{d}$, then
$$\sup\nolimits_{r\geq 0} P(Z\in B^{\epsc}\setminus B)/\epsc = \sup\nolimits_{t\geq 0} f_{\chi_{d}}(t) +O(1)= O(1), \quad d\to \infty.$$
\end{remark}

\begin{lemma}
\label{lemma:moment_conditions}

\begin{enumerate}[wide, 
labelwidth=!, 
labelindent=0pt, 
label={\Roman*.}]
\item Let $K=4$ and $X_{i}$ be a random vector in $\R^{d}$ with $\E |X_{i}^{\otimes 4}|<\infty$ and p.d. $\Var X_{i}$. Then there exists  an approximating distribution $Y_{i}$ satisfying \eqref{eq:moments_2} such that the smallest eigenvalue of $\Sigmaz$ corresponding to the normal part of the convolution $Y_{i}$ equals to an arbitrary predetermined number between $0$ and the smallest eigenvalue of $\Var X_{i}$.
\item Now let $K$ be an arbitrary integer number $\geq 3$. 
Suppose that random vector $X_{i}$ is supported in a closed set $A\subseteq \R^{d}$. Let also $\E |X_{i}^{\otimes K+1}|<\infty$ and $\Var X_{i}$ be p.d. Then the existence of the an approximating distribution $Y_{i}$ satisfying \eqref{eq:moments_2} is guaranteed either by continuity of $X_{i}$ or by a sufficiently large cardinality of $X_{i}$'s support, namely, if $X_{i}$ has a discrete probability distribution supported on $M$ points in $\R^{d}$ such that each coordinate of $X_{i}$ is supported on at least $m$ points in $\R$, it is sufficient to require $M\geq 1+(K+1)m^{d-1}$.
\end{enumerate}
\end{lemma}
\section{Proofs for Sections \ref{sect:multivar} and  \ref{sect:symm}}
\label{sect:proof_2}

The following statement is used in the proofs of main results. Inequality \eqref{ineq:H_ortog} was also derived in our earlier paper \citep{Zhilova2020supp} (Lemma A.5).
\begin{lemma}
\label{lemma:symm}
Let $A=\{a_{i_{1},\dots, i_{k}}: 1\leq i_{1},\dots, i_{k}\leq d \} \in {\R^{d}}^{\otimes k}$ be a symmetric tensor, which means that elements $a_{i_{1},\dots, i_{k}}$ of $A$ are invariant with respect to permutations of indices $\{i_{1},\dots, i_{k}\}$. It holds
\begin{align}
\label{ineq:H_ortog_A}
\Bigl| \int\nolimits_{\R^{d}}\inner{\phid^{(k)}(\xv)}{A}
d\xv\Bigr|
&\leq  \sqrt{k!}\|A\|_{\mathrm{F}} \leq \sqrt{k!}\|A\| d^{(k-1)/2}.
\end{align}

For any integer $k\geq 0$
\begin{equation}
\label{ineq:H_ortog}
\Bigl| \int\nolimits_{\R^{d}}\phid^{(k)}(\xv)
\gamma^{k}d\xv\Bigr|\leq  \sqrt{k!}\|\gamma\|^{k}\quad \forall \gamma\in \R^{d}.
\end{equation}
\end{lemma}
\begin{proof}[Proof of Lemma \ref{lemma:symm}]
Rodrigues' formula for the Hermite polynomials $\phid^{(j)}(x)=(-1)^{j}H_{j}(\xv)\phid(x)$,  orthogonality of the Hermite polynomials (see \cite{Grad1949noteHermite}), and H\"older's inequality imply
\begin{align*}
&\Bigl| \int\nolimits_{\R^{d}}\inner{\phid^{(k)}(\xv)}{A}
d\xv\Bigr|=
\Bigl| \int\nolimits_{\R^{d}} \inner{H_{k}(\xv)}{A}\phid(x)
d\xv\Bigr|
\\&\leq
\Bigl| \int\nolimits_{\R^{d}} |\inner{H_{k}(\xv)}{A}|^{2}\phid(x)
d\xv\Bigr|^{1/2}
\leq\sqrt{k!}  \|A\|_{\mathrm{F}} \leq \sqrt{k!}\|A\| d^{(k-1)/2}.
\end{align*}
The last inequality  follows from the relations between the Frobenius and the operator norms, see \cite{Wang2017operator}. Inequality 
\eqref{ineq:H_ortog} is obtained by taking $A=\gamma^{\otimes k}$.
\end{proof}
\begin{proof}[Proof of Theorem \ref{theorem:first}]
Take $\St_{n} \define n^{-1/2}\sum\nolimits_{i=1}^{n}Y_{i}$ for i.i.d. 
\begin{equation}
\label{def:Y_i}
Y_{i}\define Z_{i}+\alpha_{i}\tilde{X}_{i}
\end{equation}
 such that 
\begin{equation}
\label{eq:Y_moments}
\E (Y_{i}^{\otimes j})=\E (X_{i}^{\otimes j})\quad \forall j\in\{1,2,3\},
\end{equation}
where $\tilde{X}_{i}$ is an i.i.d. copy of $X_{i}$, $Z_{i}\sim\mathcal{N}(0,\beta^{2}  \Var X)$ is independent from $\tilde{X}_{i}$, $\beta$ is an arbitrary number in $(0,1)$, scalar random variable $\alpha_{i}$ is independent from all other random variables and is such that  
\begin{equation}
\label{eq:alpha_moments}
\E\alpha_{i}=0, \quad \E\alpha_{i}^2= \betau^{2}\define 1-\beta^{2}, \quad\E\alpha_{i}^3=1, \quad \E\alpha_{i}^{4}=\betau^{4}+\betau^{-2}.
\end{equation}
 Such random variable $\alpha_{i}$ exists $\forall \beta\in(0,1)$ due to the criterion by \cite{Curto1991recursiveness}.  Indeed, the existence of a probability distribution with moments \eqref{eq:alpha_moments} is equivalent to the p.s.d. property of the corresponding Hankel matrix, which is ensured by taking the smallest admissible $\E\alpha_{i}^{4}\define \betau^{4}+\betau^{-2}$:
 \begin{align*}
 \det \begin{pmatrix}
 1&0&\betau^{2}\\
 0& \betau^{2}& 1\\
\betau^{2} & 1&\E\alpha_{i}^{4}
 \end{pmatrix} &=  \betau^{2} \{\E\alpha_{i}^{4}- \betau^{4}-\betau^{-2}\}\geq 0.
\end{align*}
Denote for $i\in\{1,\dots,n\}$ the standardized versions of the terms in $Y_{i}$ as follows:
$$\Ut_{i}\define n^{-1/2}\beta^{-1}\alpha_{i}\Sigma^{-1/2}\tilde{X}_{i},\quad \Zt_{i}\define n^{-1/2}\beta^{-1}\betau Z_{0,i},$$ where $Z, Z_{0,i}\sim \mathcal{N}(0,\Id_{d})$ are  i.i.d. and  independent from all other random variables. Let also 
$$\st_{l}\define \sum\nolimits_{i=1}^{l-1}\Zt_{i}+\sum\nolimits_{i=l+1}^{n}\Ut_{i}$$ 
for $l=1,\dots,n$, where the sums are taken equal zero if index $i$ runs beyond the specified range.  Random vectors $\st_{l}$ are independent from  $\Ut_{l}$, $\Zt_{l}$  and 
$\st_{l}+ \Zt_{l}=\st_{l+1}+ \Ut_{l+1},\quad l=1,\dots,n-1.$
This allows to construct the telescopic sum between $f(\tsum_{i=1}^{n}\Ut_{i})$  and $f(\tsum_{i=1}^{n}\Zt_{i})$ for an arbitrary function $f: \R^{d}\mapsto \R $. Indeed, 
 $f(\tsum_{i=1}^{n}\Ut_{i}) -f(\tsum_{i=1}^{n}\Zt_{i})
=f(\st_{1}+\Ut_{1}) -f(\st_{n}+\Zt_{n})
=\sum\nolimits_{l=1}^{n}\{f(\st_{l}+\Ut_{k}) -f(\st_{l}+\Zt_{l})\}.$

In the proofs we use also Taylor's  formula \eqref{eq:Taylor}: for a sufficiently smooth function $f:\R^{d}\mapsto \R$ and $x,h\in\R^{d}$
\begin{align}
\label{eq:Taylor}
\begin{split}
f(x+h)&=\sum\nolimits_{j=0}^{s}f^{(j)}(x)h^{j}/j! 
+\E(1-\tau)^{s}f^{(s+1)}(x+\tau h)h^{s+1}/{s!},
\end{split}
\end{align}
where $\tau\sim U(0,1)$ is independent from all other random variables.

Let  $B\in\BCl$ and $B^{\prime}\define \{x\in \R^{d}: \beta\Sigma^{1/2} x\in B\}$ denote the transformed version of $B$ after the standardization. It holds 
\begin{align}
\nonumber
&\P(\St_{n} \in B)-\P(\Zsigma \in B)
\\
\nonumber
&=\E\Bigl\{ \P\bigl(Z+\tsum_{l=1}^{n}\Ut_{l}\in B^{\prime}\vert \{\Ut_{l}\}_{l=1}^{n}\bigr)-\P\bigl(Z+\tsum_{l=1}^{n}\Zt_{l}\in B^{\prime}\vert \{\Zt_{l}\}_{l=1}^{n}\bigr)\Bigr\}
\\
\label{eq:tel}
& =
\sum\nolimits_{l=1}^{n}\E \int\nolimits_{B^{\prime}}\phid(t-\st_{l}-\Ut_{l})-\phid(t-\st_{l}-\Zt_{l}) dt
\\
\label{eq:Taylorapp}
&=
\sum\nolimits_{l=1}^{n}\sum\nolimits_{j=0}^{3}(j!)^{-1}(-1)^{j}\E \int\nolimits_{B^{\prime}}\phid^{(j)}(t-\st_{l})\Ut_{l}^{j}-\phid^{(j)}(t-\st_{l})\Zt_{l}^{j} dt
+R_{4}
\\&=-6^{-1} \sum\nolimits_{l=1}^{n}\E \int\nolimits_{B^{\prime}}\phid^{(3)}(t-\st_{l})\Ut_{l}^{3} dt +R_{4}
\label{eq:proof1}
\\
&
\leq
\beta^{-3} 6^{-1/2} n^{-1/2}R_{3}
 + |R_{4}|.
\label{eq:proof1_R4}
 \end{align}
 \eqref{eq:tel} is obtained by applying the telescopic sum to the (conditional) probability set function, that is $f(x)\coloneqq \P\bigl(Z+x\in B^{\prime}\bigr)=\int_{B^{\prime}}\phid(t-x)dt$. 
 
 In \eqref{eq:Taylorapp} we expand $\phid(t-\st_{l}+x)$ around 0 w.r.t. $x=-\Ut_{l}$ and $x=-\Zt_{l}$ using Taylor's formula \eqref{eq:Taylor} for $s=4$. \eqref{eq:proof1} follows from mutual independence between  $\Ut_{l}$, $\Zt_{l}$, $\st_{l}$, from Fubini's theorem, and from the property that the first two moments of $\Ut_{l}$, $\Zt_{l}$ are equal two each other:
 \begin{equation}
 \label{eq:moments_1}
 \E(\Ut_{l}^{\otimes j})= \E(\Zt_{l}^{\otimes j}),\quad \text{for } j=1,2.
 \end{equation}
 The term $R_{3}$ is specified as follows: 
 \begin{gather}
 \label{def:R3}
R_{3}\define \sup\nolimits_{B\in \BCl}\{- 6^{-1/2}\inner {\E (\Sigma^{-1/2} X_{1})^{\otimes 3}}{V_{B}}\}
\end{gather}
for
$V_{B} \define n^{-1}\sum\nolimits_{l=1}^{n}\E \int\nolimits_{B^{\prime}}\phid^{(3)}(t-\st_{l}) dt$. This representation of $R_{3}$ follows from mutual independence between  $\Ut_{l}$, $\Zt_{l}$, $\st_{l}$ and from the expressions for the third order moments of $\Ut_{l}$, $\Zt_{l}$:
 \begin{equation}
 \label{eq:moments_2}
\E\{ (\alpha_{i}\tilde{X}_{i})^{\otimes 3}\} = \E (\tilde{X}_{i}^{\otimes 3}),\quad  \E (\Zt_{i}^{\otimes 3})=0.
 \end{equation}
Furthermore, by inequalities \eqref{ineq:H_ortog_A} in Lemma \ref{lemma:symm}, it holds for the summands in $R_{3}$:
\begin{align}
\nonumber
6^{-1/2}\E \int\nolimits_{B^{\prime}}\inner{\phid^{(3)}(t-\st_{l})}{\E (\Sigma^{-1/2} X_{1})^{\otimes 3}} dt
&\leq 
\|\E (\Sigma^{-1/2} X_{1})^{\otimes 3}\|_{\mathrm{F}}
\\&\leq \|\E (\Sigma^{-1/2} X_{1})^{\otimes 3}\| d.
\label{ineq:F_op}
 \end{align}
In addition, let $N$ denote the number of nonzero elements in $\E (\Sigma^{-1/2} X_{1})^{\otimes 3}$. If $\|\E (\Sigma^{-1/2} X_{1})^{\otimes 3}\|_{\max}\leq m_{3}$, then $|R_{3}|\leq m_{3}\sqrt{N}$, which can be smaller than the term in \eqref{ineq:F_op} if $N\leq d^{2}$. If all the coordinates of $X_{1}$ are mutually independent, then $N\leq d$.

Below we consider $R_{4}$ equal to the remainder term in expansions \eqref{eq:Taylorapp}
\begin{align}
\label{eq:R4_in}
R_{4} &\define \sum\nolimits_{l=1}^{n}6^{-1}\E (1-\tau)^{3}\int\nolimits_{B^{\prime}}\phid^{(4)}(t-\st_{l}-\tau\Ut_{l})\Ut_{l}^{4}dt
\\
\nonumber
&\quad-\sum\nolimits_{l=1}^{n}6^{-1}\E (1-\tau)^{3}\int\nolimits_{B^{\prime}}\phid^{(4)}(t-\st_{l}-\tau\Zt_{l})\Zt_{l}^{4}dt
\\
\label{ineq:R4_H}
& 
\leq   (n\beta^{4} \sqrt{4!})^{-1}\{ \E \|\Sigma^{-1/2}U_{1}\|^4+\betau^{4}\E\| Z_{0,1}\|^{4}\}
\\
& 
=   (n\beta^{4} \sqrt{4!})^{-1} \bigl\{\betau^{4}( \E \| \Sigma^{-1/2}  X_{1}\|^4+2d+d^{2})+\betau^{-2} \E \| \Sigma^{-1/2}  X_{1}\|^4\bigr\}.
\label{eq:R4}
\end{align}
\eqref{ineq:R4_H} follows from \eqref{ineq:H_ortog} in Lemma \ref{lemma:symm}.

Now we consider the following term
\begin{align}
\nonumber
&\bigl|\P(S_{n}\in B)-\P(\St_{n}\in B)\bigr|\\
\nonumber &\leq 
2C_{B,4} \|\Sigma^{-1}\|\bigl\{  
\beta^{-4}(\betau^{4}+\betau^{-2}+1/4)\E\|X_{1}\|^{4}
+
\|\Sigma\|^{2}(2d+d^{2})
\bigr\}^{1/2}n^{-1/2}
\\&\leq
\label{ineq:BE_bound}
2C_{B,4} \|\Sigma^{-1}\|\|\Sigma\|\bigl\{  
\beta^{-4}(\betau^{4}+\betau^{-2}+1/4)\E\|\Sigma^{-1/2}X_{1}\|^{4}
+(2d+d^{2})
\bigr\}^{1/2}n^{-1/2},
\end{align}
where constant $C_{B,4}=M(4)C_{\ell_{2}}C_{\phi,4}\geq 9.5$. This inequality follows from Theorem 2.1 in \citep{Zhilova2020}, which is based on the Berry--Esseen inequality by \cite{Bentkus2003BE}. We discuss this result in Section \ref{sect:cited}, and explain the source of the constants $M(4),C_{\ell_{2}},C_{\phi,4}$. Bounds \eqref{ineq:BE_bound}, \eqref{eq:proof1}, and \eqref{eq:R4} lead to the resulting statement.
\end{proof}
\begin{proof}[Proof of Theorem \ref{theorem:multivar_gen}] 
We begin with the proof of the second inequality in the statement.  
Here we employ and modify the arguments in the proof of Theorem \ref{theorem:first}.  
Firstly we construct the sums $\St_{n}, \St_{T,n}$; take
\begin{equation}
\label{eq:Stn_def}
\St_{n} \define n^{-1/2}\sum\nolimits_{i=1}^{n}Y_{i}, \quad \St_{T,n}\define n^{-1/2}\sum\nolimits_{i=1}^{n}Y_{T,i}
\end{equation} 
for  $Y_{i}\define Z_{i}+U_{i},\  Y_{T,i}\define Z^{\prime}_{i}+U_{T,i}$ 
such that random vectors $Z_{i}, U_{i},Z^{\prime}_{i}, U_{T,i}$ are independent from each other, 
\begin{equation}
\label{eq:Y_moments_T}
\E (Y_{i}^{\otimes j})=\E (X_{i}^{\otimes j}),\quad \E (Y_{T,i}^{\otimes j})=\E (T_{i}^{\otimes j}) \quad \forall j\in\{1,2,3\},
\end{equation}
 $Z_{i}, Z_{i}^{\prime}\sim\mathcal{N}(0,\beta^{2}\lmin^{2} \Id_{d})$ for $\lmin^{2}>0$ equal to the minimum of the smallest eigenvalues of $\Sigmax$ and $\SigmaT$. 
 \begin{equation}
 \label{def:Ui_T}
 U_{i}\define \alpha_{i}\tilde{X}_{i}+\beta \{\Sigmax-\lmin^{2}\Id_{d}\}^{1/2}Z_{0,i},\quad U_{T,i}\define \alpha_{i}^{\prime}\tilde{T}_{i}+\beta \{\SigmaT-\lmin^{2}\Id_{d}\}^{1/2}Z_{0,i}^{\prime},
 \end{equation}
where $\beta$ is an arbitrary number in $(0,1)$, $ \alpha_{i}$ and its i.i.d. copy $\alpha_{i}^{\prime}$ are taken as in \eqref{eq:alpha_moments}, $Z_{0,i}\sim \mathcal{N}(0, \Id_{d})$ is independent from all other random variables, and $\tilde{X}_{i},\tilde{T}_{i},Z_{0,i}^{\prime}$ are i.i.d. copies of ${X}_{i},{T}_{i},Z_{0,i}$ respectively. 

Similarly to \eqref{ineq:BE_bound}, by Theorem \ref{theorem:BE16},
\begin{align*}
&\bigl|\P(S_{n}\in B)-\P(\St_{n}\in B)\bigr|+\bigl|\P(S_{T,n}\in B)-\P(\St_{T,n}\in B)\bigr|
\\&\leq
4 C_{B,4}\lmin^{-2}\bigl[
2\beta^{-4}(\betau^{4}+\betau^{-2})\{\E\|X_{1}\|^{4}+\E\|T_{1}\|^{4}\}
\\&\quad
+(2d+d^{2})(1+ 2\|\Sigmax\|^{2}+2\|\SigmaT\|^{2})
\bigr]^{1/2}n^{-1/2}.
\end{align*}
Consider the term $\bigl|\P( \St_{n}\in B)-\P(\St_{T,n}\in B)\bigr|$. Denote for $i\in\{1,\dots,n\}$ the standardized versions of the terms in $Y_{i},Y_{T,i}$ as follows:
$$\Ut_{i}\define n^{-1/2}\beta^{-1}\lmin^{-1}U_{i},\quad \Ut_{T,i}\define n^{-1/2}\beta^{-1}\lmin^{-1}U_{T,i}.$$
Let also 
$$\st_{l}\define \sum\nolimits_{i=1}^{l-1} \Ut_{T,i}+\sum\nolimits_{i=l+1}^{n}\Ut_{i}$$ 
for $l=1,\dots,n$, where the sums are taken equal zero if index $i$ runs beyond the specified range.  Random vectors $\st_{l}$ are independent from  $\Ut_{l}$, $\Ut_{T,l}$  and 
$\st_{l}+ \Ut_{T,l}=\st_{l+1}+ \Ut_{l+1},\quad l=1,\dots,n-1.$

Let $B^{\prime}\define \{x\in \R^{d}: \beta\lmin x\in B\}$ denote the transformed version of $B\in\BCl$ after the standardization. It holds similarly to \eqref{eq:tel}-\eqref{eq:proof1_R4}
\begin{align}
\nonumber
&\P(\St_{n} \in B)-\P(\St_{T,n} \in B)
\\
\nonumber
&=\E\Bigl\{ \P\bigl(Z+\tsum_{l=1}^{n}\Ut_{l}\in B^{\prime}\vert \{\Ut_{l}\}_{l=1}^{n}\bigr)-\P\bigl(Z+\tsum_{l=1}^{n}\Ut_{T,l}\in B^{\prime}\vert \{\Ut_{T,l}\}_{l=1}^{n}\bigr)\Bigr\}
\\
\nonumber
& =
\sum\nolimits_{l=1}^{n}\E \int\nolimits_{B^{\prime}}\phid(t-\st_{l}-\Ut_{l})-\phid(t-\st_{l}-\Ut_{T,l}) dt
\\
\label{eq:Taylor_T}
&=
\sum\nolimits_{l=1}^{n}\sum\nolimits_{j=0}^{3}(j!)^{-1}(-1)^{j}\E \int\nolimits_{B^{\prime}}\phid^{(j)}(t-\st_{l})\Ut_{l}^{j}-\phid^{(j)}(t-\st_{l})\Ut_{T,l}^{j} dt
+R_{4}
\\\nonumber
&=2^{-1} \sum\nolimits_{l=1}^{n}\E \int\nolimits_{B^{\prime}}\inner{\phid^{(2)}(t-\st_{l})}{\Ut_{l}^{\otimes 2}-\Ut_{T,l}^{\otimes 2}} dt
\\&\quad-6^{-1} \sum\nolimits_{l=1}^{n}\E \int\nolimits_{B^{\prime}}\inner{\phid^{(3)}(t-\st_{l})}{\Ut_{l}^{\otimes3}-\Ut_{T,l}^{\otimes3}} dt +R_{4,T}
\nonumber
\\
&
\leq \beta^{-2} 2^{-1/2} \lmin^{-2}\|\Sigmax-\SigmaT\|_{\mathrm{F}}+
\beta^{-3} 6^{-1/2} n^{-1/2}R_{3,T}
 + |R_{4,T}|,
\label{ineq:one_T}
 \end{align}
where
\begin{gather}
 \label{def:R3_T}
R_{3,T}\define \sup\nolimits_{B^{\prime}\in \BCl}\{- 6^{-1/2}\lmin^{-3}\inner {\E (X_{1}^{\otimes 3})-\E (T_{1}^{\otimes 3})}{V_{3,T,B}}\}
\end{gather}
for $V_{3,T,B}\define n^{-1}\sum\nolimits_{l=1}^{n}\E \int\nolimits_{B^{\prime}}\phid^{(3)}(t-\st_{l,T}) dt.$
By \eqref{ineq:H_ortog_A} in Lemma \ref{lemma:symm}, it holds for the summands in $R_{3,T}$:
\begin{align}
\nonumber
&6^{-1/2}\lmin^{-3}\E \int\nolimits_{B^{\prime}}\inner{\phid^{(3)}(t-\st_{l})}{\E (X_{1}^{\otimes 3})-\E (T_{1}^{\otimes 3})} dt
\\&\leq 
\lmin^{-3}\|\E (X_{1}^{\otimes 3})-\E (T_{1}^{\otimes 3})\|_{\mathrm{F}}
\leq\lmin^{-3} \|\E (X_{1}^{\otimes 3})-\E (T_{1}^{\otimes 3})\| d.
\label{ineq:F_op_T}
 \end{align}
Let $N_{T}$ denote the number of nonzero elements in $\E (X_{1}^{\otimes 3})-\E (T_{1}^{\otimes 3})$. If $\lmin^{-3}\|\E (X_{1}^{\otimes 3})-\E (T_{1}^{\otimes 3})\|_{\max}\leq m_{3,T}$, then $|R_{3,T}|\leq m_{3,T}\sqrt{N_{T}}$, which can be smaller than the terms in \eqref{ineq:F_op_T} if $N_{T}\leq d^{2}$. If all the coordinates of $X_{1}$ and $T_{1}$ are mutually independent, then $N_{T}\leq d$.

Consider $R_{4,T}$ equal to the remainder term in Taylor expansions in \eqref{eq:Taylor_T}
\begin{align}
\nonumber
&R_{4,T} \define \sum\nolimits_{l=1}^{n}6^{-1}\E (1-\tau)^{3}\int\nolimits_{B^{\prime}}\phid^{(4)}(t-\st_{l}-\tau\Ut_{l})\Ut_{l}^{4}dt
\\
\nonumber
&\quad-\sum\nolimits_{l=1}^{n}6^{-1}\E (1-\tau)^{3}\int\nolimits_{B^{\prime}}\phid^{(4)}(t-\st_{l}-\tau\Ut_{T,l})\Ut_{T,l}^{4}dt
\\
\label{ineq:R4_H_T}
& 
\leq   (n\beta^{4}\lmin^{4} \sqrt{4!})^{-1}4\bigl\{ (\E \|X_{1}\|^4+\E \|T_{1}\|^4)(\betau^{4}+\betau^{-2})
\\
\nonumber
&\quad+\beta^{4}(d^{2}+2d)(\|\Sigmax\|^{2}+\|\SigmaT\|^{2}) \bigr\}.
\end{align}
\eqref{ineq:R4_H_T} follows from \eqref{ineq:H_ortog} in Lemma \ref{lemma:symm}, and from definitions \eqref{def:Ui_T} of $U_{i},U_{T,i}$. 
 Inequalities \eqref{ineq:one_T}, \eqref{ineq:F_op_T}, and \eqref{ineq:R4_H_T} lead to the resulting bound. 

The first part of the statement, for $\Var X_{1}=\Var T_{1}$, is derived similarly. Here  \eqref{def:Ui_T} is modified to
\begin{equation*}
 U_{i}\define \alpha_{i}\tilde{X}_{i},\quad U_{T,i}\define \alpha_{i}^{\prime}\tilde{T}_{i},
 \end{equation*}
and  $Z_{i}, Z_{i}^{\prime}\sim\mathcal{N}(0,\beta^{2} \Var X_{1})$, as in the proof of Theorem \ref{theorem:first}.
\end{proof}
\begin{proof}[Proof of Theorem \ref{theorem:first_HS}]
Here one can take w.l.o.g. $\|\gamma\|=1$. 
 We proceed similarly to the proof of Theorem \ref{theorem:first} above. 
 Let $\phid_{1}$ denote the p.d.f. of $\mathcal{N}(0,1)$ in $\R^{1}$, and let $\St_{n}$ be as in Theorem \ref{theorem:first}.
\begin{align}
\nonumber
&
\P( \gamma^{T}\St_{n} \leq x)-\P(\gamma^{T}\Zsigma\leq x )
\\\
\nonumber
&=\E\bigl\{ \P\bigl( \gamma^{T}\{Z+ \tsum_{l=1}^{n}\Ut_{l}\}\leq \beta^{-1}x\vert \{\Ut_{l}\}_{l=1}^{n}\bigr)
\\&\quad-\P\bigl( \gamma^{T}\{Z+\tsum_{l=1}^{n}\Zt_{l}\}\leq \beta^{-1}x\vert \{\Zt_{l}\}_{l=1}^{n}\bigr)\bigr\}
\nonumber
\\
\nonumber
& =
\sum\nolimits_{l=1}^{n}\E \int\nolimits_{-\infty}^{x/\beta}\phid_{1}(t- \gamma^{T}\{\st_{l}+\Ut_{l}\})-\phid_{1}(t- \gamma^{T}\{\st_{l}+\Zt_{l}\}) dt
\\
\nonumber
&=-6^{-1} \sum\nolimits_{l=1}^{n}\E \int\nolimits_{-\infty}^{x/\beta}\phid_{1}^{(3)}(t-\gamma^{T}\st_{l})(\gamma^{T}\Ut_{l})^{3} dt +R_{1,4}
\nonumber
\\
&
\leq
\beta^{-3} 6^{-1/2}n^{-1/2} R_{1,3}
 + |R_{1,4}|,
 \end{align}
 where 
$R_{1,3}\define \sup\nolimits_{x,\gamma}\{- 6^{-1/2} \E (\gamma^{T}\Sigma^{-1/2} X_{1})^{ 3}V_{1,x,\gamma}\}$ 
for \\$V_{1,x,\gamma}\define n^{-1}\sum\nolimits_{l=1}^{n}\E \int\nolimits_{-\infty}^{x/\beta}\phid_{1}^{(3)}(t-\gamma^{T}\st_{l}) dt \leq  \sqrt{3!}.$ 
Hence
$$|R_{1,3}|\leq  \| \E (\Sigma^{-1/2} X_{1})^{\otimes 3}\| =\sup\nolimits_{\gamma\in \R : \|\gamma\|=1} \E (\gamma^{T}\Sigma^{-1/2} X_{1})^{3}.$$
The remainder term $R_{1,4}$ is defined is follows
\begin{align}
\nonumber
&R_{1,4} \define \sum\nolimits_{l=1}^{n}6^{-1}\E (1-\tau)^{3}\int\nolimits_{-\infty}^{x/\beta}\phid_{1}^{(4)}(t-\gamma^{T}\{\st_{l}+\tau\Ut_{l}\})\Ut_{l}^{4}dt
\\
\nonumber
&\quad-\sum\nolimits_{l=1}^{n}6^{-1}\E (1-\tau)^{3}\int\nolimits_{-\infty}^{x/\beta}\phid_{1}^{(4)}(t-\gamma^{T}\{\st_{l}+\tau\Zt_{l}\})\Zt_{l}^{4}dt
\\
\nonumber
& 
=   (n\beta^{4} \sqrt{4!})^{-1}\{ \E (\gamma^{T}\Sigma^{-1/2}U_{1})^4+\betau^{4}\E( \gamma^{T}Z_{0,1})^{4}\}
\\
\nonumber
&\leq 
 (n\beta^{4} \sqrt{4!})^{-1}\{
 3\betau^{4}+ (\betau^{4}+\betau^{-2})\| \E (\Sigma^{-1/2} X_{1})^{\otimes 4}\|
 \}.
\end{align}

It holds similarly to \eqref{ineq:BE_bound} 
\begin{align}
\label{ineq:Delta_th_2}
 \Delta_{\HCl}(S_{n},\St_{n})
 & \leq
C_{H,4} \beta^{-2}\bigl\{  
(\betau^{4}+\betau^{-2}+1)\| \E (\Sigma^{-1/2} X_{1})^{\otimes 4}\|+3-3\betau^{4}
\bigr\}^{1/2}n^{-1/2},
\end{align}
where $C_{H,4}=M(4) \tilde{C}_{\phi}$, $M(4)= 9.5$. This term is analogous to  $C_{B,4}$ in \eqref{ineq:BE_bound}, it comes from Theorem \ref{theorem:BE16}, here on can take  $C_{\ell_{2}}=1$.
\end{proof}
Proof of Theorem \ref{theorem:HS_var} is analogous to the proofs of Theorems \ref{theorem:multivar_gen} and  \ref{theorem:first_HS}.

\begin{proof}[Proof of Theorem \ref{theorem:symm}] We follow the scheme of the proof of Theorem \ref{theorem:first}.
Take $$\St_{n} \define n^{-1/2}\sum\nolimits_{i=1}^{n}L_{i,}$$
 where $\{L_{i}=Z_{\Sigma_{L},i}+U_{L,i}\}_{i=1}^{n}$ are i.i.d. copies of $L =Z_{\Sigma_{L}}+U_{L}$. 
Denote for $i\in\{1,\dots,n\}$ the standardized versions of the terms in $L_{i}$ as follows:
$$\Ut_{i}\define n^{-1/2}\Sigma_{L}^{-1/2}U_{L,i},\quad \Zt_{i}\define n^{-1/2}\Sigma_{L}^{-1/2} (\Sigma-\Sigma_{L})^{1/2}Z_{0,i},$$ where $Z_{0,i}\sim \mathcal{N}(0,\Id_{d})$  are  i.i.d. and  independent from all other random variables. Without loss of generality we can assume 
\begin{align}
\label{eq:lminz}
\Sigma_{L}&= \lminz^{2}\Id_{d},
\end{align}
since due to condition \eqref{condit:L_moments} the smallest eigenvalue $\lminz^{2}$ of $\Sigma_{L}$ is positive. Let also 
$$\st_{l}\define \sum\nolimits_{i=1}^{l-1}\Zt_{i}+\sum\nolimits_{i=l+1}^{n}\Ut_{i}$$ 
for $l=1,\dots,n$, where the sums are taken equal zero if index $i$ runs beyond the specified range.  Random vectors $\st_{l}$ are independent from  $\Ut_{l}$, $\Zt_{l}$  and 
$\st_{l}+ \Zt_{l}=\st_{l+1}+ \Ut_{l+1},\quad l=1,\dots,n-1.$ 
Take arbitrary $B\in \BCl$, and let $B^{\prime}\define \{x\in \R^{d}: \Sigma_{L}^{1/2} x\in B\}$ denote the transformed version of $B$ after the standardization. Let also random vector $Z\sim \mathcal{N}(0,\Id_{d})$ be independent from all other  random variables. It holds
\begin{align}
\nonumber
&\P(\St_{n} \in B)-\P(\Zsigma \in B)
\\
\nonumber
&=\E\Bigl\{ \P\bigl(Z+\tsum_{l=1}^{n}\Ut_{l}\in B^{\prime}\vert \{\Ut_{l}\}_{l=1}^{n}\bigr)-\P\bigl(Z+\tsum_{l=1}^{n}\Zt_{l}\in B^{\prime}\vert \{\Zt_{l}\}_{l=1}^{n}\bigr)\Bigr\}
\\
\nonumber
& =
\sum\nolimits_{l=1}^{n}\E \int\nolimits_{B^{\prime}}\phid(t-\st_{l}-\Ut_{l})-\phid(t-\st_{l}-\Zt_{l}) dt
\\
\label{eq:Taylorapp_symm}
&=
\sum\nolimits_{l=1}^{n}\sum\nolimits_{j=0}^{5}(j!)^{-1}(-1)^{j}\E \int\nolimits_{B^{\prime}}\phid^{(j)}(t-\st_{l})\Ut_{l}^{j}-\phid^{(j)}(t-\st_{l})\Zt_{l}^{j} dt
+R_{6,L}
\\&=24^{-1} \sum\nolimits_{l=1}^{n}\E \int\nolimits_{B^{\prime}}\inner{\phid^{(4)}(t-\st_{l})}{\Ut_{l}^{\otimes 4} - \Zt_{l}^{\otimes 4}}dt +R_{6,L}
\label{eq:proof1_symm}
\\
&
\leq
24^{-1/2} n^{-1} \|\E \{(\Sigma_{L}^{-1/2}U_{L,1})^{\otimes 4}\} - \E\{(\Sigma_{L}^{-1/2} (\Sigma-\Sigma_{L})^{1/2}Z_{0,1})^{\otimes 4}\} \|_{\mathrm{F}} 
 + |R_{6,L}|.
\label{eq:proof1_R6L}
\\
&=
24^{-1/2} n^{-1} \lminz^{-4}\|\E (X_{1}^{\otimes 4}) - \E(\Zsigma^{\otimes 4})\|_{\mathrm{F}} 
 + |R_{6,L}|.
 \nonumber
 \end{align}
In \eqref{eq:Taylorapp_symm} we consider the 5-th order Taylor expansion of $\phid(t-\st_{l}+x)$ around 0 w.r.t. $x=-\Ut_{l}$ and $x=-\Zt_{l}$, with the error term $R_{6,L}$. \eqref{eq:proof1_symm} follows from condition \eqref{condit:poly} which implies $E (X_{i}^{\otimes j})=0$ for $j=1,3,5$. \eqref{eq:proof1_R6L} follows from  \eqref{ineq:H_ortog_A} in Lemma \ref{lemma:symm} (similarly to the bounds on $R_{3}$ and $R_{3,T}$ in \eqref{ineq:F_op}, \eqref{ineq:F_op_T}). %

The remainder term $R_{6,L}$ is specified as follows: 
\begin{align}
\nonumber
&R_{6,L} \define \sum\nolimits_{l=1}^{n}(5!)^{-1}\E (1-\tau)^{5}\int\nolimits_{B^{\prime}}\phid^{(6)}(t-\st_{l}-\tau\Ut_{l})\Ut_{l}^{6}dt
\\
\nonumber
&\quad-\sum\nolimits_{l=1}^{n}(5!)^{-1}\E (1-\tau)^{5}\int\nolimits_{B^{\prime}}\phid^{(6)}(t-\st_{l}-\tau\Zt_{l})\Zt_{l}^{6}dt
\\
\label{ineq:R6_H_L}
& 
\leq   (n^{2} \sqrt{6!})^{-1}\{ \lminz^{-6} \E \|U_{L}\|^6+\E\|(\lminz^{-2}\Sigma-\Id_{d})^{1/2} Z_{0,1}\|^{6}\}
\\& 
\leq  C_{6,L} d^{3}/n^{2}, 
\nonumber
\end{align}
where $C_{6,L}=C_{6,L}(\lminz,\Sigma, \E(U_{L}^{\otimes 6}))$, and 
\eqref{ineq:R6_H_L} follows from \eqref{ineq:H_ortog}.

Consider the term $\bigl|\P(S_{n}\in B)-\P(\St_{n}\in B)\bigr|$,  here we apply Theorem \ref{theorem:BE16} for $K=6$ and $C_{B,6}=2.9 C_{\ell_{2}}C_{\phi,6}$ (cf. \eqref{ineq:BE_bound} in the proof of Theorem \ref{theorem:first}):
\begin{align}
\nonumber
\bigl|\P(S_{n}\in B)-\P(\St_{n}\in B)\bigr|
 &\leq 
C_{B,6}\bigl\{ 
\lminz^{-6} 
\E(\|X_{1}\|^{6}+\|L_{1}\|^{6})
\bigr\}^{1/4}n^{-1/2},
\\&\leq
C_{B,6}(
\lminz^{-6} 
m_{6,sym}
)^{1/4}d^{3/4}n^{-1/2}.
\nonumber
\end{align}
\end{proof}
\begin{proof}[Proof of Proposition \ref{prop:example_d}]
\begin{align*}
\Delta_{\BCl}(S_{n},Z)  &\geq  \sup\nolimits_{x\geq 0}\bigl| P( \|S_{n}\|^{2}\leq x)- P( \|Z\|^{2}\leq x)\bigr|
\\&\geq \Delta_{L}( \|S_{n}\|^{2}, \|Z\|^{2}) 
\\&=\Delta_{L}(\|Z\|^{2}  +D_{n}, \|Z\|^{2})
\geq C d^{3/2}/n
\end{align*}
for sufficiently large $d$ and $n$, and for a generic constant $C>0$. The latter inequality follows from the definition of the  L\'evy distance, from boundedness of $D_{n}(d^{2}/n)^{-1}$ in probability and, since for a fixed $a>0$
$$\sup _{x\geq 0} \int_{x}^{x+a} f_{\chi^{2}_{d}}(t)dt  =aO(d^{-1/2}),\quad d\to \infty,$$
$f_{\chi^{2}_{d}}(t)$ denotes the p.d.f. of $\|Z\|^{2}\sim \chi^{2}_{d}$ (cf. Remark \ref{ref:AC_opt} in Section \ref{sect:cited}).
\end{proof}
\begin{proof}[Proof of Lemma \ref{lemma:appl}]
 Condition \ref{condit:poly} is fulfilled by the symmetry of $\mathcal{N}(0,\Id_{d})$ around the origin. We construct vector $L$ such  that \eqref{condit:L_moments} is fulfilled with prescribed $\lminz$. 
 Let $Y=(y_{1},\dots,y_{d})^{T}$, where $y_{j}$ are i.i.d. centered and standardized  double exponential, with p.d.f. $2^{-1/2}e^{-\sqrt{2}\abs{x}}$ for $x\in \R$. Take  
 \begin{align*}
 L&\define (1-\sqrt{2/5})^{1/2}\tilde{Z}+ (2/5)^{1/4}\{YY^{T}\}^{1/2}Z 
 \end{align*}
 for i.i.d.  $Z,\tilde{Z}\sim  \mathcal{N}(0, \Id_{d})$ independent from $Y$. Then $\E (L^{\otimes j})=0$ for j=1,3,5,  $\Var L= \Id_{d}$, and 
 $$\E (L^{\otimes 4})_{i,j,k,l}=\E (X^{\otimes 4})_{i,j,k,l}= \begin{cases}
9,& i=j= k=l,\\
1,&i=j\neq k=l,\\
0,& \text{otherwise}.
 \end{cases}$$
\end{proof}

\section{Proofs for Sections \ref{section:ES_bootstr}-\ref{sect:score}}
\label{sect:ES_boostr_proofs}
Below we cite Bernstein's inequality for sub-exponential random variables by \cite{BouchLug2013Conc} (this is a short version of Theorem 2.10 in the monograph), which is used in the proofs in this section.
\begin{theorem}[\cite{BouchLug2013Conc}]
\label{theorem:conc_book}
Let $X_{1},\dots, X_{n}$ be independent real-valued random variables. Assume that there exist positive numbers $\nu$ and $c$ such that $\sum_{i=1}^{n} \E (X_{i}^{2}) \leq \nu$ and 
$\sum_{i=1}^{n} \E (X_{i}^{q})_{+} \leq q! \nu c^{q-2}/2$ for all integers $q\geq 3$, where $x_{+}=\max(x,0)$, then for all $t>0$
$$\P\Bigl(\sum\nolimits_{i=1}^{n} (X_{i} - \E X_{i}) \geq \sqrt{2\nu t} + ct\Bigr) \leq e^{-t}.$$
\end{theorem}
We also use the following statement
\begin{lemma}
\label{lemma:tail_ineq}
Let real-valued random variables  $x,y \in \mathcal{G}(\sigma^{2})$ for some $\sigma^{2}>0$ (see definition \eqref{condit:expm1} in Section \ref{section:ES_bootstr}), then it holds for all $t>0$
\begin{align*}
\P\left(\abs{xy-\E(xy)} \geq 4\sigma^{2} ( \sqrt{8t} + t) \right)&\leq 2  e^{-t}
\end{align*}
\end{lemma}
\begin{proof}[Proof of Lemma \ref{lemma:tail_ineq}]
We show that random variable ${xy-\E(xy)}$ satisfies conditions of Theorem \ref{theorem:conc_book} and apply its statement for $n=1$. Theorem 2.1 by \cite{BouchLug2013Conc} implies that if $x\in \mathcal{G}(\sigma^{2})$, then 
\begin{align*}
\E (x^{2q}) \leq 2q! (2\sigma^{2})^{q}\quad \forall\,\text{integer } q\geq 1.
\end{align*}
Therefore, it holds by H\"older's and Jensen's inequalities
\begin{align*}
\E (xy-\E(xy))^{2} &\leq  \E (xy)^2 \leq  \{\E (x^{4}) \E (y^{4})\}^{1/2} \leq 16\sigma^{4},\\
 \E (xy-\E(xy))_{+}^{q}& \leq  \E ( |xy|^{q} )2^{q} \leq  \{\E (x^{2q}) \E (y^{2q})\}^{1/2} 2^{q}
 \leq  2q! (4 \sigma^{2})^{q}.
\end{align*}
These inequalities allow to take $\nu = 64 \sigma^{4}$ and $c=4\sigma^{2}$ in the statement of Theorem \ref{theorem:conc_book}.
\end{proof}

\begin{proof}[Proof of Theorem \ref{theorem:bootst}]
Without loss of generality, let $\mu=0$. We use the construction in the proof of Theorem \ref{theorem:multivar_gen}. 
Denote 
$$\Sigmac \define n^{-1}\tsum_{i=1}^{n} X_{i}^{\otimes 2},\quad \Sigmah\define n^{-1}{{\textstyle\sum\nolimits_{i=1}^{n}}} (X_{i}- \Xmean)^{\otimes 2}.$$
Since each coordinate of $\Xmean$ belongs to $\mathcal{G}(\sigma^{2}/n)$, it holds 
for any $t>0$
\begin{align}
\label{ineq:Xbar_norm}
\P\bigl(\|\Xmean\|^{2} \leq  2\sigma^{2}(d/n) t  \bigr) &\geq 1 -2 d e^{-t}
\end{align}
Similarly, by condition \eqref{condit:expm1}, Lemma \ref{lemma:tail_ineq}, and \eqref{ineq:Xbar_norm}, it holds for any $t>0$
\begin{align*}
\P\bigl(\| \Sigmac-\Sigma\|\leq  4\sigma^{2}(dn^{-1/2}) (\sqrt{8t} + tn^{-1/2})   \bigr) &\geq 1-(d^{2}+d)e^{-t},\\
\P\bigl(\| \Sigmah-\Sigma\|_{\mathrm{F}}\leq  2\sigma^{2}(d/n^{1/2})\{4\sqrt{2t} + 3tn^{-1/2}\}\bigr) &\geq 1-(d^{2}+3d)e^{-t}.
\end{align*}
Hence we can take   
$\lmin^{2}\define \lambda_{\min}(\Sigma)-2\sigma^{2}(d/n^{-1/2})\{4\sqrt{2t} + 3tn^{-1/2}\},$ provided that this expression is positive (this is ensured by condition \eqref{condit:lambda_n}).

It holds for the bootstrap terms included in ${V}_{T,4}$ and $v_{4}$: 
\begin{gather*}
\P\bigl( \Es\| X^{\ast}_{j}\|^{4}\leq 8(1+n^{-2})\{2\sigma^{2}(d/n) t  \}^{2}\bigr) \geq 1-2nde^{-t} ,\\
\P\bigl( \|\Sigmah\|^{2}\leq  2\|\Sigma\|^{2} + 2 \bigl\{ \sigma^{2}(d/\sqrt{n})C_{1}(t)\bigr\}^{2}
\bigr) \geq 1-(d^{2}+3d)e^{-t}.
\end{gather*}

Now we consider the Frobenius norm of the difference between the 3-d order moments:
  \begin{align*}
 & \|\Es({\Xs_{j}}^{\otimes 3})-\E( X_{1}^{\otimes 3})\|_{\mathrm{F}}
 \\& \leq   \|n^{-1}\tsum_{i=1}^{n}X_{i}^{\otimes 3}-\E( X_{1}^{\otimes 3})\|_{\mathrm{F}} +2\|\bar{X}\|^{3}+3\|\bar{X}\|\|\Sigmac \|_{\mathrm{F}}
 \\&\leq
  \|n^{-1}\tsum_{i=1}^{n}X_{i}^{\otimes 3}-\E( X_{1}^{\otimes 3})\|_{\mathrm{F}}
  +2 \{2\sigma^{2}(d/n) t\}^{3/2} 
  \\&\quad+3\sqrt{ 2\sigma^{2}(d/n) t} \{\|\Sigma\|_{\mathrm{F}}+4\sigma^{2}(dn^{-1/2}) (\sqrt{8t} + tn^{-1/2})  \}
 \end{align*}
 with probability $\geq 1-(d^{2}+3d)e^{-t}$.
    \begin{align*}
      &\|n^{-1}\tsum_{i=1}^{n}X_{i}^{\otimes 3}-\E( X_{1}^{\otimes 3})\|_{\mathrm{F}}
\\&\leq
   \|n^{-1}\tsum_{i=1}^{n}X_{i}^{\otimes 3}-n^{-1}\tsum_{i=1}^{n}X_{i}\otimes\Sigma\|_{\mathrm{F}} 
     + \|\Xbar\|\|\Sigma \|_{\mathrm{F}} 
      +\|\E( X_{1}^{\otimes 3})\|_{\mathrm{F}}      
     \\&\leq    
      \|n^{-1}\tsum_{i=1}^{n}X_{i}^{\otimes 3}-n^{-1}\tsum_{i=1}^{n}X_{i}\otimes\Sigma\|_{\mathrm{F}} 
     + \{  2\sigma^{2}(d/n) t\}^{1/2}  \|\Sigma\|_{\mathrm{F}} 
      +\|\E( X_{1}^{\otimes 3})\|_{\mathrm{F}}.
 \end{align*}
For the term $\|n^{-1}\tsum_{i=1}^{n}X_{i}^{\otimes 3}-n^{-1}\tsum_{i=1}^{n}X_{i}\otimes\Sigma\|_{\mathrm{F}} $, we have 
     \begin{align*}
       &\|n^{-1}\tsum_{i=1}^{n}X_{i}^{\otimes 3}-n^{-1}\tsum_{i=1}^{n}X_{i}\otimes\Sigma\|_{\mathrm{F}} 
    \\ &\leq \max\nolimits_{1\leq i\leq n} \|X_{i}\|_{\max} 4\sigma^{2}(d^{3/2}n^{-1/2}) (\sqrt{8t} + tn^{-1/2})  
      \\ &\leq 4\sqrt{2t}\sigma^{3}(d^{3/2}n^{-1/2}) (\sqrt{8t} + tn^{-1/2})  
      \end{align*}
with probability $\geq1-2 dn e^{-t}.$
Collecting all bounds together, we derive
     \begin{align*}
&\sup\nolimits_{B\in \BCl}| \P(\Sbar\in B)- \Ps(\Snast\in B )|
\\&\leq 
(\sqrt{2}\beta^{2}\lmin^{2})^{-1}\bigl\{  \sigma^{2}(d/\sqrt{n})C_{1}(t) \bigr\}
\\&
+(\sqrt{6}\beta^{3} \lmin^{3})^{-1}
\Bigl[ 
2 \{2\sigma^{2}(d/n) t\}^{3/2} 
+ 3\sqrt{ 2\sigma^{2}(d/n) t} \{4\sigma^{2}(dn^{-1/2}) (\sqrt{8t} + tn^{-1/2})  \}
\\&\quad + 4\sqrt{2t}\sigma^{3}(d^{3/2}n^{-1/2}) (\sqrt{8t} + tn^{-1/2})  
+4\{  2\sigma^{2}(d/n) t\}^{1/2}  \|\Sigma\|_{\mathrm{F}} +\|\E( X_{1}^{\otimes 3})\|_{\mathrm{F}}
\Bigr]n^{-1/2}
\\
&  
+ 4\sqrt{2}C_{B,4}\lmin^{-2}\Bigl\{  
h_{1}(\beta) [ \E \|X_{1}\|^{4}+  8(1+n^{-2})\{2\sigma^{2}(d/n) t  \}^{2} ]
\\&
\hspace{0.5cm}
+(d^{2}+2d)( 3\|\Sigma\|^{2} + 2 \bigl\{ \sigma^{2}(d/\sqrt{n})C_{1}(t)\bigr\}^{2}+1/2)
\Bigr\}^{1/2}n^{-1/2}
\\&+   2(\sqrt{6}\lmin^{4})^{-1}\Bigl\{h_{1}(\beta) [ \E \|X_{1}\|^{4}+8(1+n^{-2})\{2\sigma^{2}(d/n) t  \}^{2} ]
\\& 
\hspace{2cm}
+(d^{2}+2d)[ 3\|\Sigma\|^{2} +  2 \bigl\{ \sigma^{2}(d/\sqrt{n})C_{1}(t)\bigr\}^{2}] \Bigr\}n^{-1}
      \end{align*}
      with probability $\geq1- (2 dn +d^{2}+3d)e^{-t},$ which leads to the resulting statement with $t\define t_{\ast}$.
\end{proof}



Proposition \ref{prop:ellip2} follows from Theorem \ref{theorem:bootst}, with the error term $\delta_{W}$. Denote $\lmin^{2}\define
 \lambda_{\min}(W^{1/2}\Sigma W^{1/2})- \sigma_{W}^{2}(d/\sqrt{n})C_{1\ast}$ and $X_{0,1}\overset{d}{=}X_{1}-\E X_{1}$, then
 \begin{align}
\label{def:deltaW}
\delta_{W} 
&\define 
(\sqrt{2}\beta^{2}\lmin^{2})^{-1}\bigl\{  \sigma_{W}^{2}(d/\sqrt{n})C_{1\ast} \bigr\}
\\
\nonumber
&\quad
+(\sqrt{6}\beta^{3} \lmin^{3})^{-1}
\Bigl[ 
4\sigma_{W} \sqrt{2 dn^{-2} t_{\ast}}
\{ \|W^{1/2}\Sigma W^{1/2}\|_{\mathrm{F}}+\sigma_{W}^{2}(d/n)t_{\ast}\}
\\
\nonumber
&\quad+
\sigma_{W}^{2}d^{3/2}n^{-1} C_{2\ast}\{1+3n^{-1/2}\} +\|\E(W^{1/2} X_{0,1})^{\otimes 3}\|_{\mathrm{F}} n^{-1/2}
\Bigr]
\\
\nonumber
&  
\quad+  4\sqrt{2}C_{B,4}\lmin^{-2}\Bigl\{  
h_{1}(\beta) \bigl[ \E \|W^{1/2}X_{0,1}\|^{4}+  8(1+n^{-2})\{2\sigma_{W}^{2}(d/n) t_{\ast}  \}^{2} \bigr]
\\
\nonumber
&
\quad+(d^{2}+2d)( 3\|W^{1/2}\Sigma W^{1/2} \|^{2} + 2 \bigl\{ \sigma_{W}^{2}(d/\sqrt{n})C_{1\ast}\bigr\}^{2}+1/2)
\Bigr\}^{1/2}n^{-1/2}
\\
\nonumber
&
\quad+    2(\sqrt{6}\lmin^{4})^{-1}\Bigl\{h_{1}(\beta) \bigl[ \E \|W^{1/2}X_{0,1}\|^{4}+8(1+n^{-2})\{2 \sigma_{W}^{2}(d/n) t_{\ast}  \}^{2} \bigr]
\\
\nonumber
& 
\quad+(d^{2}+2d)\bigl[ 3\|W^{1/2} \Sigma W^{1/2}\|^{2} +  2 \bigl\{ \sigma_{W}^{2}(d/\sqrt{n})C_{1\ast}\bigr\}^{2}\bigr] \Bigr\}n^{-1}+n^{-1}.
\end{align}
Theorem \ref{theorem:score1} follows from  Theorem \ref{theorem:bootst}, with the following error term. Denote $\Sigma_{s}\define I(\theta^{\prime})/n$, $X_{1}= \partial \log p(y_{1};\theta^{\prime})/{\partial \theta^{\prime}}$, 
 $\lmin^{2}\define  \lambda_{\min}(\Sigma_{s})- \sigma_{s}^{2}(d/\sqrt{n})C_{1\ast}$,
then
\begin{align}
\label{def:deltaR}
\delta_{R}&
\define 
(\sqrt{2}\beta^{2}\lmin^{2})^{-1}\bigl\{\sigma_{s}^{2}(d/\sqrt{n})C_{1\ast} \bigr\}
\\
\nonumber
&\quad
+(\sqrt{6}\beta^{3} \lmin^{3})^{-1}
\Bigl[ 
4\sigma \sqrt{2 dn^{-2} t_{\ast}}
\{ \|\Sigma_{s}\|_{\mathrm{F}}+\sigma_{s}^{2}(d/n)t_{\ast}\}
\\
\nonumber
&\quad+
\sigma_{s}^{2}d^{3/2}n^{-1} C_{2\ast}\{1+3n^{-1/2}\} +\|\E( X_{1}^{\otimes 3})\|_{\mathrm{F}} n^{-1/2}
\Bigr]
\\
\nonumber
&  
\quad+  4\sqrt{2}C_{B,4}\lmin^{-2}\Bigl\{  
h_{1}(\beta) \bigl[ \E \|X_{1}\|^{4}+  8(1+n^{-2})\{2\sigma_{s}^{2}(d/n) t_{\ast} \}^{2} \bigr]
\\
\nonumber
&
\quad+(d^{2}+2d)( 3\|\Sigma_{s}\|^{2} + 2 \bigl\{ \sigma_{s}^{2}(d/\sqrt{n})C_{1\ast}\bigr\}^{2}+1/2)
\Bigr\}^{1/2}n^{-1/2}
\\
\nonumber
&
\quad+    2(\sqrt{6}\lmin^{4})^{-1}\Bigl\{h_{1}(\beta)\bigl[ \E \|X_{1}\|^{4}+8(1+n^{-2})\{2\sigma_{s}^{2}(d/n) t_{\ast}  \}^{2} \bigr]
\\
\nonumber
& 
\quad+(d^{2}+2d)\bigl[ 3\|\Sigma_{s}\|^{2} +  2 \bigl\{ \sigma^{2}(d/\sqrt{n})C_{1\ast}\bigr\}^{2}\bigr] \Bigr\}n^{-1}.
\end{align}

Theorem \ref{theorem:score2} follows from Theorem \ref{theorem:first}.

\bibliographystyle{imsart-nameyear}
\bibliography{references.bib}

\end{document}